\documentclass[final,leqno]{siamltex}
\usepackage{amsmath,amsfonts,amssymb,amscd}

\usepackage{graphicx}
\graphicspath{{pics/}}

\def\0{\phantom{0}}
\def\d{\partial}
\def\twoVec(#1,#2){\left(\begin{matrix}#1\\#2\end{matrix}\right)}
\def\grad{\nabla}
\def\div{\nabla\cdot}
\def\Div{\textrm{div}}
\def\Curl{\textrm{curl}}
\def\Grad{\textrm{grad}}
\def\spn{\textrm{span}}

\def\Re{\mathbb R}
\def\Po{\mathbb P}
\def\Supp{\mathbb S}
\def\Bu{\mathbb B}

\def\a{{\mathbf a}}
\def\F{{\mathbf F}}
\def\V{{\mathbf V}}

\def\v{{\mathbf v}}
\def\u{{\mathbf u}}
\def\x{{\mathbf x}}
\def\bfpsi{\pmb\psi}
\def\bfsigma{\pmb\sigma}

\def\cA{{\cal A}}

\def\cC{{\cal C}}
\def\cDS{{\cal{DS}}}
\def\cI{{\cal I}}
\def\cL{{\cal L}}
\def\cN{{\cal N}}
\def\cO{{\cal O}}
\def\cP{{\cal P}}
\def\cS{{\cal S}}
\def\cT{{\cal T}}

\def\myStrut{\vphantom{\int^H}}
\def\Line(#1,#2)(#3,#4){\qbezier(#1,#2)(#1,#2)(#3,#4)}
\def\longlongrightarrow%
{\DOTSB\protect\relbar\protect\joinrel\protect\relbar\protect\joinrel\protect\relbar\protect\joinrel\rightarrow}
\def\hooklongrightarrow{\DOTSB\lhook\joinrel\longrightarrow}

\title{Direct Serendipity and Mixed Finite Elements\\on Convex Quadrilaterals%
\thanks{This work was supported by the U.S.~National Science Foundation under grants DMS-1418752
and DMS-1720349.}}

\author{
  Todd Arbogast\thanks{University of Texas at Austin;
  Department of Mathematics, C1200; Austin, TX 78712-1202 and
  Institute for Computational Engineering and Sciences, C0200;
  Austin, TX 78712-1229 (arbogast@ices.utexas.edu)}
\and
 Zhen Tao\thanks{University of Texas at Austin;
 Institute for Computational Engineering and Sciences, C0200;
 Austin, TX 78712--1229 (taozhen.cn@gmail.com)}
}

\begin{document}

\maketitle

\date{\today}

\begin{abstract}
  The classical serendipity and mixed finite element spaces suffer from poor approximation on
  nondegenerate, convex quadrilaterals.  In this paper, we develop \emph{direct serendipity} and
  \emph{direct mixed} finite element spaces, which achieve optimal approximation properties and have
  minimal local dimension.  The set of local shape functions for either the serendipity or mixed
  elements contains the full set of scalar or vector polynomials of degree $r$, respectively,
  defined directly on each element (i.e., not mapped from a reference element).  Because there are
  not enough degrees of freedom for global $H^1$ or $H(\textrm{div})$ conformity, exactly two
  supplemental shape functions must be added to each element.  The specific choice of supplemental
  functions gives rise to different families of direct elements. These new spaces are related
  through a de Rham complex.  For index $r\ge1$, the new families of serendipity spaces $\cDS_{r+1}$
  are the precursors under the curl operator of our direct mixed finite element spaces $\V_r$, which
  can be constructed to have full or reduced $H(\textrm{div})$ approximation properties.  One choice
  of direct serendipity supplements gives the precursor of the recently introduced Arbogast-Correa
  spaces [SIAM J.\ Numer.\ Anal., 54 (2016), pp.~3332--3356]. Other \emph{fully} direct serendipity
  supplements can be defined without the use of mappings from reference elements, and these give
  rise in turn to \emph{fully} direct mixed spaces.  Numerical results are presented to illustrate
  the properties of the new spaces.
\end{abstract}




\pagestyle{myheadings} \thispagestyle{plain}
\markboth{Todd Arbogast and Zhen Tao}
{Direct Serendipity and Mixed Finite Elements}


\section{Introduction}\label{sec:Intro}

On a rectangle $\hat E$, serendipity finite elements $\cS_r(\hat E)$
\cite{Strang_Fix_1973,Ciarlet_1978,Brenner_Scott_1994} and Brezzi-Douglas-Marini mixed finite
elements BDM$_r(\hat E)$ \cite{BDM_1985} appear in the periodic table of the finite elements as
given by Arnold and Logg \cite{Arnold_Logg_2014_periodicTable} (where they are denoted
$\cS_r\Lambda^0$ and $\cS_r\Lambda^1$, respectively). They should be studied together, since they
are related by a de Rham complex \cite{Arnold_Falk_Winter_2006, Arnold_2013, AFW_2010_feec}
\begin{equation}\label{eq:deRhamBDM}
\Re \hooklongrightarrow \cS_{r+1}(\hat E) \overset{\Curl\,}{\longlongrightarrow}
 \textrm{BDM}_r(\hat E) \overset{\Div\,}{\longlongrightarrow} \Po_{r-1}(\hat E) \longrightarrow 0,
\end{equation}
which implies that $\textrm{BDM}_r(\hat E)=\Curl\,\cS_{r+1}(\hat E)\oplus\x\Po_{r-1}(\hat E)$, where
$\Po_{s}(\Hat E)$ are polynomials of degree $s$.  Over a rectangular mesh, the serendipity elements
merge into $H^1$ conforming spaces of scalar functions, and the BDM elements merge into
$H(\Div)=\big\{\v\in(L^2)^2:\div\v\in L^2\big\}$ conforming spaces of vector functions.  In this
paper, we define new (we call them \emph{direct}) serendipity and mixed finite elements on a general
nondegenerate, convex quadrilateral $E$.  These new elements generalize the complex
\eqref{eq:deRhamBDM}, and they maintain optimal order approximation properties while possessing
minimal local dimension.

The serendipity finite elements on rectangles $\cS_r(\hat E)$, especially the 8-node biquadratic
($r=1$) and the 12-node bicubic ($r=2$) ones, have been well studied for many years. They appear in
almost any introductory reference on finite elements,
e.g.,~\cite{Strang_Fix_1973,Ciarlet_1978,Brenner_Scott_1994}, and they are provided by software
packages both in academia~\cite{DHJKLLS_2003_fenics} and
industry~\cite{Hibbitt_Karlsson_Sorensen_2001_abaqus}.  Compared with the full tensor product
Lagrange finite elements $\Po_{r,r}(\hat E)$, serendipity finite elements use fewer degrees of
freedom, and they are usually more efficient.  It was not until recently, however, that a general
definition of the serendipity finite element spaces of arbitrary order on rectangles in any space
dimension was given by Arnold and Awanou~\cite{Arnold_Awanou_2011, Arnold_Awanou_2014} (see also
\cite{Gillette_Kloefkorn_2018_trimmed}).

The serendipity finite element spaces work very well on computational meshes of rectangular
elements, but it is well known that their performance is degraded on quadrilaterals when the space
is mapped from a rectangle, when $r\ge2$.  This is not the case for tensor product Lagrange finite
elements~\cite{Lee_Bathe_1993,Kaliakin_2001,ABF_2002}.  To be more precise, mapped serendipity
elements of index $r$ do not approximate to optimal order $r+1$ on $E$, but the image of the full
space of tensor product polynomials $\Po_{r,r}(\hat E)$ maintains accuracy on $E$.  We note that
Rand, Gillette, and Bajaj \cite{Rand_Gillette_Bajaj_2014} recently introduced a new family of
Serendipity finite elements based on generalized barycentric coordinates of index $r=2$ that is
accurate to order three on any convex, planar polygon.  A generalization to any order of approximation
was given by Floater and Lai \cite{Floater_Lai_2016}, but on quadrilaterals, they require
$\dim\Po_r+r$ shape functions, which is more than the minimal required when $r>2$.

There are many families of mixed finite elements on rectangles, beginning with those of Raviart and
Thomas \cite{Raviart_Thomas_1977} and generalized by N\'ed\'elec \cite{Nedelec_1980}.  These and the
BDM$_r$ finite elements are extended to quadrilaterals using the Piola transform \cite{Thomas_1977,
  Raviart_Thomas_1977}.  For most spaces, this creates a consistency error and consequent loss of
approximation of the divergence \cite{Thomas_1977, Brezzi_Fortin_1991, ABF_2005,
  Boffi_Brezzi_Fortin_2013, Arbogast_Correa_2016}.

The construction of mixed finite elements on quadrilaterals that maintain optimal order accuracy is
considered in many papers.  Most address only low order cases (see, e.g., \cite{Shen_1994,
  Shen_1992_phd, Boffi_Kikuchi_Schoberl_2006, Bochev_Ridzal_2008, Duan_Liang_2004,
  Kwak_Pyo_2011}). The exceptions we are aware of are the families of finite elements of Arnold,
Boffi, and Falk (ABF$_r$(E))~\cite{ABF_2005}, Siqueira, Devloo, and Gomes
\cite{Siqueira_Devloo_Gomes_2013}, and Arbogast and Correa (AC$_r(E)$ and AC$_r^{\textrm{red}}(E)$)
\cite{Arbogast_Correa_2016}.  The ABF elements are defined for rectangles and extended to
quadrilaterals in the usual way (i.e., by mapping via the Piola transformation). They rectify the
problem of poor divergence approximation by including more degrees of freedom in the space, so that
approximation properties are maintained after Piola mapping.  The spaces of
\cite{Siqueira_Devloo_Gomes_2013} also involve the Piola map, but in a unique way.  They also add
shape functions to their space to obtain accuracy. The AC elements use a different strategy.  These
elements are defined by using vector polynomials directly on the element (i.e., without being
mapped) and supplemented by two vector shape functions defined on a reference square and mapped via
Piola.  The AC spaces have minimal local dimension.

In this paper, we introduce new families of \emph{direct} serendipity and mixed finite elements that
have optimal approximation properties and maintain minimal local dimension. They are \emph{direct}
in the sense that the shape functions contain a full set of polynomials defined directly on the
element, as in the AC spaces.  Because there are not enough degrees of freedom to achieve $H^1$ or
$H(\Div)$ conformity over meshes of quadrilaterals, two supplemental functions need to be added to
each element, as is done for the AC spaces.

The families of direct serendipity elements have the same number of degrees of freedom as the
corresponding classical serendipity element, and they take the form
\begin{equation}\label{eq:serendipityForm}
\cDS_r(E) = \Po_r(E)\oplus\Supp_r^\cDS(E),\quad r\geq2.
\end{equation}
Each family is defined by the choice of the two supplemental functions spanning $\Supp_r^\cDS(E)$.  We
give a very general and explicit construction for these supplements.  They can be defined directly
on $E$, or they can be defined on $\hat E$ and mapped to $E$.

There are two classes of families of direct mixed elements, which correspond to full and reduced
$H(\Div)$-approximation.  For index $r$, a vector function is approximated to order $r+1$ accuracy,
but the divergence of the vector is approximated to order $r$ or $r-1$ for full and reduced
$H(\Div)$-approximation, respectively.  Each class of direct mixed elements has the same optimal
number of degrees of freedom as the AC elements of that class.  They take a form similar to
\eqref{eq:serendipityForm}, which is
\begin{equation}
\V_r^{\textrm{red}}(E) = \Po_r^2(E)\oplus\Supp_r^\V(E),\quad
\V_r^{\textrm{full}}(E) = \V_r^{\textrm{red}}(E)\oplus\x\tilde\Po_r(E),\quad r\geq1,
\end{equation}
where $\tilde\Po_r$ are homogeneous polynomials of degree~$r$.  Again, each family is defined by the
choice of the two supplemental functions spanning $\Supp_r^\V(E)$.

The serendipity and mixed families are related by de Rham theory:
\begin{equation}
\Curl\,\Supp_{r+1}^\cDS(E)=\Supp_r^\V(E).
\end{equation}
We define one family of direct serendipity elements that is the precursor of the full and reduced AC
spaces.  We also define many \emph{fully} direct serendipity elements that use no mappings to define
$\Supp_r^{\cDS}(E)$, which in turn generate new full and reduced \emph{fully} direct mixed spaces
that use no mappings whatsoever. Moreover, a second de Rham complex involving the gradient and curl
operators provides new $H(\Curl)=\big\{\v\in(L^2)^2:\Curl\,\v\in L^2\big\}$ elements as well.

We set some basic notation in the next section.  For any index $r\ge2$, we construct new families of
direct serendipity elements in Sections~\ref{sec:DS}--\ref{sec:mappingSupp} for which the
supplements either do not or do involve mappings, respectively.  Through the de Rham theory, these
lead to the AC and new direct mixed elements in Section~\ref{sec:deRham}.  We discuss the stability
and convergence properties of the new elements in Section~\ref{sec:properties}, and give some
numerical results illustrating their performance in Section~\ref{sec:numerics}.  A summary of our
results and conclusions, as well as $H(\Curl)$ elements, are given in the final section.


\section{Some notation}\label{sec:notation}

Let $\Po_r(\omega)$ denote the space of polynomials of degree up to $r$ on $\omega\subset\Re^d$,
where $d=0$ (a point), $1$, or~$2$. Recall that
\begin{equation}\label{eq:dimPo}
\dim\Po_r(\Re^d) = \twoVec(r+d,d) = \frac{(r+d)!}{r!\,d!}.
\end{equation}
Let $\tilde\Po_r(\omega)$ denote the space of homogeneous polynomials of degree $r$ on $\omega$.  Then
\begin{equation}\label{eq:dimTildePo}
\dim\tilde\Po_r(\Re^d) = \twoVec(r+d-1,d-1) = \frac{(r+d-1)!}{r!\,(d-1)!}.
\end{equation}

Let the element $E\subset\Re^2$ be a closed, nondegenerate, convex quadrilateral.  By nondegenerate,
we mean that $E$ does not degenerate to a triangle, line segment, or point. Let the reference
element $\hat E$ be $[-1,1]^2$. Define the bilinear and bijective map $\F_{\!E}:\hat E\to E$ that
maps the vertices of $\hat E$ to those of $E$.  We identify ``vertical'' and ``horizontal'' pairs of
opposite edges and number them consecutively as shown in Figure~\ref{fig:numbering}.  Let $\nu_i$
denote the unit outer normal to edge $i$ (denoted $e_i$), $i=1,2,3,4$, and identify the vertices as
$\x_{v,13}=e_1\cap e_3$, $\x_{v,14}=e_1\cap e_4$, $\x_{v,23}=e_2\cap e_3$, and
$\x_{v,24}=e_2\cap e_4$.

\begin{figure}[ht]\centering
\setlength\unitlength{3.2pt}
\begin{picture}(46,32)(-6,-6)\small
%
\thicklines
\put(0,0){\line(0,1){20}}
\put(0,0){\line(1,0){20}}
\put(20,0){\line(0,1){20}}
\put(0,20){\line(1,0){20}}
\put(0,10){\circle*{1}}\put(0,10){\vector(-1,0){4}}
\put(10,0){\circle*{1}}\put(10,0){\vector(0,-1){4}}
\put(20,10){\circle*{1}}\put(20,10){\vector(1,0){4}}
\put(10,20){\circle*{1}}\put(10,20){\vector(0,1){4}}
\put(-4.7,8.5){\makebox(0,0){$\hat\nu_1$}}
\put(11.2,-5){\makebox(0,0){$\hat\nu_3$}}
\put(25,8.3){\makebox(0,0){$\hat\nu_2$}}
\put(8,25){\makebox(0,0){$\hat\nu_4$}}
\put(0,0){\circle*{1}}\put(-2,-2){\makebox(0,0){$(-1,-1)$}}
\put(20,0){\circle*{1}}\put(22,-2){\makebox(0,0){$(1,-1)$}}
\put(20,20){\circle*{1}}\put(22,22){\makebox(0,0){$(1,1)$}}
\put(0,20){\circle*{1}}\put(-2,22){\makebox(0,0){$(-1,1)$}}
\put(10,10){\makebox(0,0){$\hat E$}}
\put(1.7,14){\makebox(0,0){$\hat e_1$}}
\put(6.5,-1.6){\makebox(0,0){$\hat e_3$}}
\put(18.4,7){\makebox(0,0){$\hat e_2$}}
\put(14,21.6){\makebox(0,0){$\hat e_4$}}
\put(36,14){\makebox(0,0){$\overset{\text{\normalsize$\F_{\!E}$}}
{-\!\!\!-\!\!\!\longrightarrow}$}}
\end{picture}\qquad
\raisebox{6pt}{\begin{picture}(34.5,32)(-5,-6)\small
%
\thicklines
\put(0,0){\line(1,3){4}}
\put(0,0){\line(1,0){24}}
\put(24,0){\line(-1,4){4}}
\put(4,12){\line(4,1){16}}
\put(2,6){\circle*{1}}\put(2,6){\vector(-3,1){4.74}}
\put(12,0){\circle*{1}}\put(12,0){\vector(0,-1){5}}
\put(22,8){\circle*{1}}\put(22,8){\vector(4,1){4.85}}
\put(12,14){\circle*{1}}\put(12,14){\vector(-1,4){1.21}}
\put(-4,7.5){\makebox(0,0){$\nu_1$}}
\put(10.8,-5.5){\makebox(0,0){$\nu_3$}}
\put(28.4,9){\makebox(0,0){$\nu_2$}}
\put(10.8,20){\makebox(0,0){$\nu_4$}}
\put(0,0){\circle*{1}}\put(-1.8,-1.6){\makebox(0,0){$\x_{v,13}$}}
\put(24,0){\circle*{1}}\put(25.6,-1.6){\makebox(0,0){$\x_{v,23}$}}
\put(20,16){\circle*{1}}\put(21.6,17.6){\makebox(0,0){$\x_{v,24}$}}
\put(4,12){\circle*{1}}\put(2.4,13.6){\makebox(0,0){$\x_{v,14}$}}
\put(12,7){\makebox(0,0){$E$}}
\put(4.6,8.8){\makebox(0,0){$e_1$}}
\put(8,-1.2){\makebox(0,0){$e_3$}}
\put(21.5,4){\makebox(0,0){$e_2$}}
\put(16,16.5){\makebox(0,0){$e_4$}}
\end{picture}}
\caption{A reference element $\hat E=[-1,1]^{2}$ and quadrilateral $E$, with edges $\hat e_i$ and $e_i$,
outer unit normals $\hat\nu_i$ and $\nu_i$, and vertices $(-1,-1)$ and $\x_{v,13}$, etc., respectively.
\label{fig:numbering}}
\end{figure}
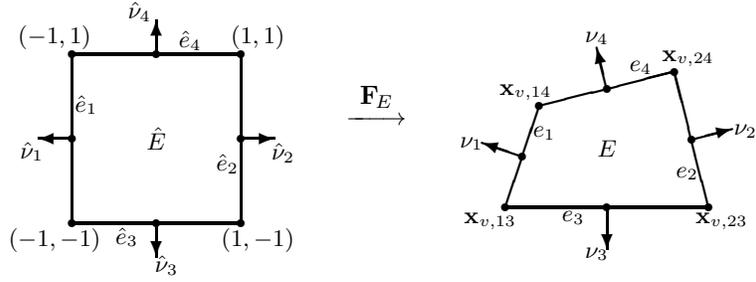

We define the linear polynomial $\lambda_i(\x)$ giving the distance of $\x\in\Re^2$ to edge~$e_i$
in the normal direction as
\begin{align}
\label{eq:lambda_i}
\lambda_i(\x) &= - (\x-\x_i)\cdot\nu_i, \quad i=1,2,3,4,
\end{align}
where $\x_i\in e_i$ is any point on the edge. If $\x$ is in the interior of $E$, these functions
are strictly positive, and each vanishes on the edge which defines it.

We denote by $F_{\!E}^0$ the map taking a function $\hat\phi$ defined on $\hat E$ to a function $\phi$
defined on $E$ by the rule
\begin{equation}\label{eq:defF0}
\phi(\x) = F_{\!E}^0(\hat\phi)(\x) = \hat\phi(\hat\x),
\end{equation}
where $\x=\F_{\!E}(\hat\x)$.  We denote by $\F_{\!E}^1$ the Piola map taking a vector function
$\hat\bfpsi$ defined on $\hat E$ to a vector function $\bfpsi$ defined on $E$ by the rule
\begin{equation}\label{eq:defF1}
\bfpsi(\x) = \frac1{J_E}DF_{\!E}(\hat\x)\,\hat\bfpsi(\hat\x),
\end{equation}
where $DF_{\!E}(\hat\x)$ is the Jacobian matrix of $\F_{\!E}$ and $J_E$ is its absolute determinant.

Recall Ciarlet's definition~\cite{Ciarlet_1978} of a finite element.

\begin{definition}[Ciarlet 1978]\label{defn:ciarlet}
Let 
\begin{enumerate}
\item[$1.$] $E\subset \Re^d$ be a bounded closed set with nonempty interior and a Lipschitz continuous
  boundary,
\item[$2.$] $\cP$ be a finite-dimensional space of functions on $E$, and
\item[$3.$] $\cN = \{ N_1, N_2,\ldots, N_k \}$ be a basis for $\cP'$.
\end{enumerate}
Then $(E, \cP, \cN)$ is called a \emph{finite element.}
\end{definition}

Our task is to define the \emph{shape functions} $\cP$ and the \emph{degrees of freedom} (DoFs)
$\cN$. The DoFs give a basis for $\cP'$ provided that we have unisolvence of the shape functions
(i.e., for $\phi\in\cP$, $N_j(\phi)=0$ for all $j$ implies that $\phi=0$). To achieve optimal
approximation properties, we will require that $\cP\supset\Po_r(E)$ for each index~$r$.  That is,
the polynomials will be directly included within the function space, and hence we call our new
finite elements \emph{direct serendipity} and \emph{direct mixed} elements.

Let $\Omega\subset\Re^2$ be a polygonal domain, and let $\cT_h$ be a conforming finite element
partition or mesh of $\Omega$ into nondegenerate, convex quadrilaterals of maximal diameter $h>0$.
To obtain approximation properties globally, we need to assume that the mesh is uniformly shape
regular \cite[pp.~104--105]{Girault_Raviart_1986}, which means the following.  For any $E\in\cT_h$,
denote by $T_i$, $i=1,2,3,4$, the subtriangle of $E$ with vertices being three of the four vertices
of $E$. Define the parameters
\begin{align}
\label{eq:hE}
h_E &= \text{diameter of }E, \\
\label{eq:rhoE}
\rho_E &= 2\,\min_{1\leq i\leq 4}\{ \text{diameter of largest circle inscribed in }T_i \}.
\end{align}
Uniform shape regularity of the meshes means that there exists $\sigma_*>0$ such that the ratio
$\displaystyle{\rho_E}/{h_E}\geq \sigma_*>0$ for all $E\in\cT_h$, where $\sigma_*$ is
independent of $\cT_h$.

The DoFs must be defined so that the shape functions on adjoining elements merge together.  For
serendipity spaces, we want the global space to reside in $H^1(\Omega)$, so the elements must merge
continuously across each edge $e$.  For mixed spaces, the vector variable must lie in
$H(\Div;\Omega)$, which means that the normal components (fluxes) of the vectors on an edge $e$ in
adjacent elements must be continuous.


\section{Fully direct serendipity elements in two space dimensions}\label{sec:DS}

It is shown in~\cite{ABF_2002} that when $d=2$, the convergence of the linear serendipity finite
element space ($r=1$) does not degenerate on quadrilaterals. The parametric serendipity element
$\cS_1(E)$ is the tensor product space of bilinear functions $\Po_{1,1}(\hat E)$ on $\hat E$ mapped
to $E$ by $F_{\!E}^0$, and, in fact,
\begin{align}\label{eq:S1}
\cS_1(E) &= \spn\{F_{\!E}^0(1),F_{\!E}^0(\hat x),F_{\!E}^0(\hat y),F_{\!E}^0(\hat x\hat y)\}\\
\nonumber
&= \spn\{1,x,y,F_{\!E}^0(\hat x\hat y)\}
= \Po_1(E)\oplus\spn\{F_{\!E}^0(\hat x\hat y)\}
\end{align}
has the form of a direct serendipity space. Therefore, we only develop our new direct serendipity
finite elements $\cDS_r(E)$ for indices $r \geq 2$.

Our dual objectives are that $\Po_r(E)\subset\cDS_r(E)$ and that shape functions on adjoining
elements merge continuously, i.e., so the space over $\Omega$ satisfies
$\cDS_r(\Omega)\subset H^1(\Omega)$.  These objectives require us to consider the lower dimensional
geometric objects within $E$ (as in \cite{Arnold_Awanou_2011}).  The minimal number of DoFs
associated to each lower dimensional object must correspond to the dimension of the polynomials that
restrict to that object.  These numbers are given in Table~\ref{tab:geometricDecomp}.  A
quadrilateral has 4 vertices, 4 edges, and one cell of dimension 0, 1, and 2, respectively. Each
vertex requires $\dim\Po(\Re^0)=1$ DoF, each edge requires $\dim\Po_{r-2}(\Re)=r-1$ DoFs (not
counting the vertices), and each cell requires $\dim\Po_{r-4}(\Re^2)=\twoVec(r-2,2)$ DoFs (not
counting the edges and vertices).  The total number of DoFs is then $D_r$, where
\begin{align*}
D_r = 4 + 4(r-1) + \frac12(r-2)(r-3) = \frac12(r+2)(r+1) + 2 = \dim\Po_r(E) + 2,
\end{align*}
and so to define $\cDS_r(E)$, we will supplement $\Po_r(E)\subset\cDS_r(E)$ with the span of two
functions.  We have many choices for the supplemental functions, the span of which is denoted
$\Supp_r^\cDS(E)$.  Each choice gives rise to a distinct family of direct serendipity elements of
index $r\geq2$; that is, the shape functions ($\cP$ in Definition~\ref{defn:ciarlet}) are
\begin{equation}\label{eq:cDS=P+S}
\cDS_r(E)=\Po_r(E)\oplus\Supp_r^\cDS(E).
\end{equation}
In this section, we develop supplemental spaces that are unmapped (i.e., these new serendipity
spaces are fully direct---even the supplements are defined directly on $E$).

\begin{table}[ht]
  \caption{
  Geometric decomposition and degrees of freedom (DoFs) associated to each
  geometric object of a quadrilateral for a serendipity element of index
  $r \geq 2$.\label{tab:geometricDecomp}}
\centerline{\begin{tabular}{ccccc}
Dimension & Object  & Object  & DoFs per  & Total  \\
                 &  Name &  Count &   Object &  DoFs \\
\hline
0 & vertex & 4 & 1 & 4\\
1 & edge & 4 & $r-1$ & $4(r-1)$ \\
2 & cell & 1 & $\frac12(r-2)(r-3)$ & $\frac12(r-2)(r-3)$\\
\end{tabular}}
\end{table}

We define the DoFs ($\cN$ in Definition~\ref{defn:ciarlet}) as a set of nodal functionals
$N_j$ defined at a nodal point $\x_{n,j}$, i.e.,
\begin{equation}\label{eq:nodalFunctionals}
\cN = \{N_j : N_j(\phi) = \phi(\x_{n,j})\text{ for all }\phi(\x),\ j=1,2,\ldots,D_r\}.
\end{equation}
As depicted in Figure~\ref{fig:nodalDoFs}, for vertex DoFs, the nodal points are exactly the
vertices $\x_{v,13}$, $\x_{v,14}$, $\x_{v,23}$, and $\x_{v,24}$ of~$E$.  For edge DoFs, the nodal
points plus vertices are equally distributed on each edge. There are $r-1$ nodal points on the
interior of each egde, which can be denoted $\x_{e_i,j}$, $j=1,\ldots,r-1$ for nodal points that lie
on edge $e_i$, $i=1,2,3,4$. The interior cell DoFs can be set, for example, on points of a
triangle~$T$ strictly inside $E$, where the set of nodal points is the same as the nodes of the
Lagrange element of order $r-4$ on the triangle~$T$.

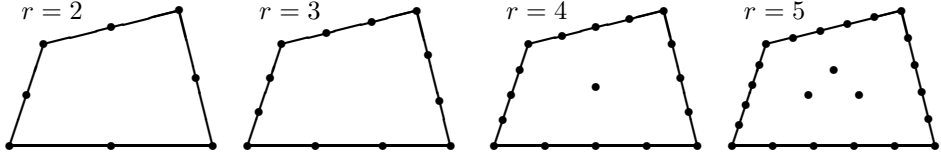
\begin{figure}[htbp]\centering{\setlength\unitlength{3.2pt}
\begin{picture}(25,18)(-0.5,-0.5)
\thicklines
\put(5,16){\makebox(0,0){$r=2$}}
\put(0,0){\line(1,3){4}}
\put(0,0){\line(1,0){24}}
\put(24,0){\line(-1,4){4}}
\put(4,12){\line(4,1){16}}
\put(2,6){\circle*{1}}
\put(12,0){\circle*{1}}
\put(22,8){\circle*{1}}
\put(12,14){\circle*{1}}
\put(0,0){\circle*{1}}
\put(24,0){\circle*{1}}
\put(20,16){\circle*{1}}
\put(4,12){\circle*{1}}
%
\end{picture}\quad
\begin{picture}(25,18)(-0.5,-0.5)
\thicklines
\put(5,16){\makebox(0,0){$r=3$}}
\put(0,0){\line(1,3){4}}
\put(0,0){\line(1,0){24}}
\put(24,0){\line(-1,4){4}}
\put(4,12){\line(4,1){16}}
\put(1.333,4){\circle*{1}}
\put(2.667,8){\circle*{1}}
\put(8,0){\circle*{1}}
\put(16,0){\circle*{1}}
\put(22.667,5.333){\circle*{1}}
\put(21.333,10.667){\circle*{1}}
\put(9.333,13.333){\circle*{1}}
\put(14.667,14.667){\circle*{1}}
\put(0,0){\circle*{1}}
\put(24,0){\circle*{1}}
\put(20,16){\circle*{1}}
\put(4,12){\circle*{1}}
%
\end{picture}
\quad
\begin{picture}(25,18)(-0.5,-0.5)
\thicklines
\put(5,16){\makebox(0,0){$r=4$}}
\put(0,0){\line(1,3){4}}
\put(0,0){\line(1,0){24}}
\put(24,0){\line(-1,4){4}}
\put(4,12){\line(4,1){16}}
\put(1,3){\circle*{1}}
\put(2,6){\circle*{1}}
\put(3,9){\circle*{1}}
\put(6,0){\circle*{1}}
\put(12,0){\circle*{1}}
\put(18,0){\circle*{1}}
\put(23,4){\circle*{1}}
\put(22,8){\circle*{1}}
\put(21,12){\circle*{1}}
\put(8,13){\circle*{1}}
\put(12,14){\circle*{1}}
\put(16,15){\circle*{1}}
\put(12,7){\circle*{1}}
\put(0,0){\circle*{1}}
\put(24,0){\circle*{1}}
\put(20,16){\circle*{1}}
\put(4,12){\circle*{1}}
%
\end{picture}\quad
\begin{picture}(25,18)(-0.5,-0.5)
\thicklines
\put(5,16){\makebox(0,0){$r=5$}}
\put(0,0){\line(1,3){4}}
\put(0,0){\line(1,0){24}}
\put(24,0){\line(-1,4){4}}
\put(4,12){\line(4,1){16}}
\put(0.8,2.4){\circle*{1}}
\put(1.6,4.8){\circle*{1}}
\put(2.4,7.2){\circle*{1}}
\put(3.2,9.6){\circle*{1}}
\put(4.8,0){\circle*{1}}
\put(9.6,0){\circle*{1}}
\put(14.4,0){\circle*{1}}
\put(19.2,0){\circle*{1}}
\put(20.8,12.8){\circle*{1}}
\put(21.6,9.6){\circle*{1}}
\put(22.4,6.4){\circle*{1}}
\put(23.2,3.2){\circle*{1}}
\put(7.2,12.8){\circle*{1}}
\put(10.4,13.6){\circle*{1}}
\put(13.6,14.4){\circle*{1}}
\put(16.8,15.2){\circle*{1}}
\put(12,9){\circle*{1}}
\put(9,6){\circle*{1}}
\put(15,6){\circle*{1}}
\put(0,0){\circle*{1}}
\put(24,0){\circle*{1}}
\put(20,16){\circle*{1}}
\put(4,12){\circle*{1}}
\end{picture}
}
\caption{ 
The nodal points for the DoFs of the direct serendipity finite element for small $r$.
\label{fig:nodalDoFs}}
\end{figure}


\subsection{Vertices}\label{sec:vertexDoFs}

For the vertices, $r\ge2$, so we can define the shape functions
\begin{equation}\label{eq:shape-vertices}
\begin{alignedat}2
\phi_{v,13}(\x) &= \lambda_2(\x)\lambda_4(\x),&\quad\phi_{v,14}(\x) &= \lambda_2(\x)\lambda_3(\x),\\
\phi_{v,23}(\x) &= \lambda_1(\x)\lambda_4(\x),&\quad\phi_{v,24}(\x) &= \lambda_1(\x)\lambda_3(\x).
\end{alignedat}
\end{equation}
These four functions are clearly linearly independent and unisolvent with respect to the vertex
DoFs.  All other shape functions will be defined so as to vanish at the vertices, so these four will
be independent of the rest.


\subsection{Interior cell}\label{sec:cell}

For the entire cell $E$, we need interior shape functions only when $r\ge4$ (recall
Table~\ref{tab:geometricDecomp}). We let the shape functions be defined by
\begin{equation}\label{eq:shape-cell}
\big\{\phi_{E,j}(\x) : j=1,\ldots,\tfrac12(r-2)(r-3)\big\} = \lambda_1\lambda_2\lambda_3\lambda_4\Po_{r-4}.
\end{equation}
These shape functions are linearly independent and vanish if the cell DoFs vanish.  Moreover, these
functions vanish on all four edges.  Therefore, if unisolvent shape functions can be defined for the
edge DoFs, then the set of all our shape functions will be unisolvent for the entire set of DoFs.


\subsection{Edges}\label{sec:edges}

We define distinct families of fully direct serendipity elements depending on the choice of the two
supplemental functions used.  These will be defined by a choice of four functions, which are
oriented ``horizontally'' or ``vertically,'' in the sense that their zero sets are horizontal or
vertical (as oriented by the bilinear map $\F_{\!E}$, see Figure~\ref{fig:numbering}).  Two of the
functions are linear polynomials, denoted $\lambda_H$ and $\lambda_V$. The other two functions
should be bounded, and they are denoted $R_V$ and $R_H$. The supplemental space is then defined as
\begin{equation}
\label{eq:supplementSpace}
\Supp_r^{\cDS}(E) = \spn\{\lambda_3\lambda_4\lambda_H^{r-2}R_V,\lambda_1\lambda_2\lambda_V^{r-2}R_H\}.
\end{equation}

The choice of the linear function $\lambda_H$ is based on the edges $e_1$ and $e_2$. As shown in
Figure~\ref{fig:lambda-h}, let $\cL_1$ and $\cL_2$ be the infinite lines containing the edges $e_1$
and $e_2$, respectively. When $e_1$ and $e_2$ are parallel, the only requirement for the choice of
$\lambda_H$ is that its zero line intersects both $\cL_1$ and $\cL_2$.  When $e_1$ and $e_2$ are not
parallel, $\cL_1$ and $\cL_2$ intersect in a point $\x_{12}$.  Then the only requirements for the
choice of $\lambda_H$ is that $\lambda_H(\x_{12})\ne0$ and that its zero line intersects both
$\cL_1$ and $\cL_2$ on the half-lines emanating from $\x_{12}$ and either containing $e_1$ and
$e_2$, respectively, or not containing $e_1$ and $e_2$, respectively (i.e., the zero line of
$\lambda_H$ intersects the lines containing $e_1$ and $e_2$ either above or below $\x_{12}$).
To be more precise in the case when $e_1$ and $e_2$ are not parallel, we
can expand $\lambda_H\in\Po_1(\Re^2)$ in the basis defined by $\{1,\lambda_1,\lambda_2\}$, so there are
constants $\alpha_H$, $\beta_H$, and $\gamma_H$ such that
\begin{equation}
\label{eq:lambda-h}
\lambda_H(\x) = \alpha_H\lambda_1(\x) + \beta_H\lambda_2(\x) + \gamma_H
 = -(\x - \x_H)\cdot(\alpha_H\nu_1+\beta_H\nu_2),\quad e_1\nparallel e_2,
\end{equation}
where $\x_H$ is any point on the zero line. The requirements are that $\alpha_H$, $\beta_H$, and
$\gamma_H$ are nonzero and that $\alpha_H$ and $\beta_H$ have the same sign.  Without loss of
generality, we may assume that $\alpha_H$ and $\beta_H$ are positive.  In a similar way, $\lambda_V$
is chosen to intersect the lines extending $e_3$ and $e_4$, and when they are not parallel, either
strictly to the left or right of the intersection point $\x_{34}$. When $e_3$ and $e_4$ are not
parallel,
\begin{equation}
\label{eq:lambda-v}
\lambda_V(\x) = \alpha_V\lambda_1(\x) + \beta_V\lambda_2(\x) + \gamma_VSee
 = -(\x - \x_V)\cdot(\alpha_V\nu_1+\beta_V\nu_2),\quad e_3\nparallel e_4,
\end{equation}
where $\x_V$ is any point on the zero line, $\alpha_V>0$, $\beta_V>0$, and $\gamma_V\ne0$.  We
remark that a simple choice is to take
\begin{equation}\label{eq:simpleLambdaHV}
\lambda_H^{\text{simple}} = \lambda_3 - \lambda_4
\quad\text{and}\quad
\lambda_V^{\text{simple}} = \lambda_1 - \lambda_2.
\end{equation}

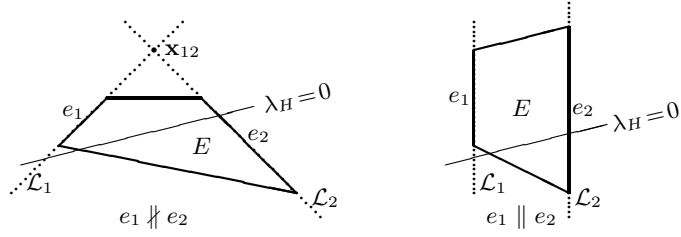
\begin{figure}
\centering{\setlength\unitlength{1.8pt}\begin{picture}(70,45)(9,2)\small
%
\thicklines
\multiput(10,10)(1.0,1.0){37}{\circle*{0.2}}
\multiput(75,5)(-1.0,1.0){42}{\circle*{0.2}}
\put(20,20){\line(1,1){10}}
\put(70,10){\line(-1,1){20}}
\put(20,20){\line(5,-1){50}}
\put(30,30){\line(1,0){20}}
\put(40,40){\circle*{1}}\put(46,40){\makebox(0,0){$\x_{12}$}}
\put(23,27){\makebox(0,0){$e_1$}}
\put(16,12){\makebox(0,0){$\cL_1$}}
\put(62,22){\makebox(0,0){$e_2$}}
\put(76,9){\makebox(0,0){$\cL_2$}}
\put(50,20){\makebox(0,0){$E$}}
\thinlines
\put(12,16){\line(4,1){49}}
\put(70,30){\rotatebox{13}{\makebox(0,0){$\lambda_H\!=\!0$}}}
\put(40,5){\makebox(0,0){$e_1\nparallel e_2$}}
\end{picture}\qquad\qquad\begin{picture}(50,50)(4,2)\small
%
\thicklines
\multiput(10,10)(0,1.0){37}{\circle*{0.2}}
\multiput(30,5)(0,1.0){46}{\circle*{0.2}}
\put(10,20){\line(0,1){20}}
\put(30,10){\line(0,1){35}}
\put(10,20){\line(2,-1){20}}
\put(10,40){\line(4,1){20}}
\put(7,30){\makebox(0,0){$e_1$}}
\put(14,12){\makebox(0,0){$\cL_1$}}
\put(33,27.5){\makebox(0,0){$e_2$}}
\put(34,9){\makebox(0,0){$\cL_2$}}
\put(20,28){\makebox(0,0){$E$}}
\thinlines
\put(4,16){\line(4,1){34}}
\put(46,26){\rotatebox{13}{\makebox(0,0){$\lambda_H\!=\!0$}}}
\put(20,5){\makebox(0,0){$e_1\parallel e_2$}}
\end{picture}}
\caption{Illustration of the zero lines of $\lambda_H$ and the point $\x_{12}$, if it exists.\label{fig:lambda-h}}
\end{figure}

The functions $R_V$ and $R_H$ are defined to satisfy the properties
\begin{alignat}3
\label{eq:RVe}
R_V(\x)|_{e_1} &= -\eta_V&&\quad\text{and}\quad&R_V(\x)|_{e_2} &= \xi_V,\\
\label{eq:RHe}
R_H(\x)|_{e_3} &= -\eta_H&&\quad\text{and}\quad&R_H(\x)|_{e_4} &= \xi_H,
\end{alignat}
for some positive constants $\eta_V$, $\xi_V$, $\eta_H$, and $\xi_H$.  For example, one choice is to let
\begin{align}
\label{eq:rational-v}
R_V^{\text{simple}}(\x) &= \frac{\lambda_1(\x) - \lambda_2(\x)}{\xi_V^{-1}\lambda_1(\x) + \eta_V^{-1}\lambda_2(\x)},\\
\label{eq:rational-h}
R_H^{\text{simple}}(\x) &= \frac{\lambda_3(\x) - \lambda_4(\x)}{\xi_H^{-1}\lambda_3(\x) + \eta_H^{-1}\lambda_4(\x)}
\end{align}
(note that the denominators do not vanish on $E$).

We now define the shape functions associated with the edge DoFs. Let
\begin{align}
\label{eq:lambda12}
\lambda_{12} &= \begin{cases}
\alpha_H\xi_V\lambda_1-\beta_H\eta_V\lambda_2, & e_1\nparallel e_2,\\
\xi_V\lambda_1 - \eta_V\lambda_2, & e_1\parallel e_2,
\end{cases}\\
\label{eq:lambda34}
\lambda_{34} &= \begin{cases}
\alpha_V\xi_H\lambda_3-\beta_V\eta_H\lambda_4, & e_3\nparallel e_4,\\
\xi_H\lambda_3 - \eta_H\lambda_4, & e_3\parallel e_4,
\end{cases}
\end{align}
which also have horizontal and vertical zero sets, respectively.  There are $2(r-1)$ shape functions
associated to the edges $e_1$ and $e_2$, and they are
\begin{alignat}2
\label{eq:shape-h}
\phi_{H,j}(\x) &= \lambda_3(\x)\lambda_4(\x)\lambda_H^{j}(\x), &&\quad j=0,1,\ldots,r-2,\\
\label{eq:shape-hpm}
\phi_{H,r-1+j}(\x) &=\lambda_3(\x)\lambda_4(\x)\lambda_{13}(\x)\lambda_H^{j}(\x), &&\quad j=0,1,\ldots,r-3,\\
\label{eq:shape-hr}
\phi_{H,2r-3}(\x) &=\lambda_3(\x)\lambda_4(\x)R_V(\x)\lambda_H^{r-2}(\x).
\end{alignat}
In a similar way, we define shape functions associated with edges $e_3$ and $e_4$ to be
\begin{alignat}2
\label{eq:shape-v}
\phi_{V,j}(\x) &= \lambda_1(\x)\lambda_2(\x)\lambda_V^{j}(\x), &&\quad j=0,1,\ldots,r-2,\\
\label{eq:shape-vpm}
\phi_{V,r-1+j}(\x) &=\lambda_1(\x)\lambda_2(\x)\lambda_{24}(\x)\lambda_V^{j}(\x), &&\quad j=0,1,\ldots,r-3,\\
\label{eq:shape-vr}
\phi_{V,2r-3}(\x) &=\lambda_1(\x)\lambda_2(\x)R_H(\x)\lambda_V^{r-2}(\x).
\end{alignat}
The edge shape functions are regular polynomials of degree~$r$ except the last two functions in each
direction, which may be rational functions, for example.  However, all shape functions restrict to
polynomials of degree~$r$ on the edges.


\subsection{Unisolvence}

The space of shape functions is
\begin{align}\label{eq:shapeSpace}
\cDS_{r}(E) 
&= {\rm{span}}\big\{ \phi_{v,13}(\x), \phi_{v,14}(\x), \phi_{v,23}(\x), \phi_{v,24}(\x),\\
\nonumber
&\qquad\qquad\phi_{H,j}(\x), \phi_{V,j}(\x)\ (j=0,1,\ldots,2r-3),\\
\nonumber
&\qquad\qquad\phi_{E,k}(\x)\ (k=1,\ldots,\tfrac12(r-2)(r-3))\big\}\\
\nonumber
&\subset \Po_r(E)\oplus\Supp_r^\cDS(E).
\end{align}
In this subsection, we show the unisolvence of the degrees of freedom, which will then show that
$\cDS_{r}(E)=\Po_r(E)\oplus\Supp_r^\cDS(E)$ and complete the requirements of Ciarlet's
Definition~\ref{defn:ciarlet} for $\cDS_r(E)$ to be a well defined finite element.  But first, we
require a lemma.

\begin{lemma}\label{lem:equ-space}
Let $k\geq1$. For any $\eta_V>0$ and $\xi_V>0$, there exists a function $\tilde R_V$, defined by
\eqref{eq:specialRV} and \eqref{eq:specialRVparallel}, with the properties \eqref{eq:RVe} such
that the two function spaces
\begin{align*}
\cA_k^1 &= {\rm{span}}\{ 1,\lambda_H,\ldots,\lambda_H^k, \lambda_{12},\lambda_H\lambda_{12},
                      \ldots, \lambda_H^{k-1}\lambda_{12},\lambda_H^k\tilde R_V\},\\
\cA_k^2 &= {\rm{span}}\big\{\{1,\lambda_H,\ldots,\lambda_H^k\}\otimes \{1,\tilde R_V\}\big\}
\end{align*}
are identical. Moreover, for any $\eta_H>0$ and $\xi_H>0$, there exists an $\tilde R_H$, defined by
\eqref{eq:specialRH}, with the properties \eqref{eq:RHe} such that the two function spaces
\begin{align*}
\cA_k^3 &= {\rm{span}}\{ 1,\lambda_V,\ldots,\lambda_V^k, \lambda_{34},\lambda_V\lambda_{34},
                     \ldots, \lambda_V^{k-1}\lambda_{34},\lambda_V^k\tilde R_H \},\\
\cA_k^4 &= {\rm{span}}\big\{\{1,\lambda_V,\ldots,\lambda_V^k\}\otimes\{1,\tilde R_H\}\big\}
\end{align*}
\end{lemma}
are identical.

\begin{proof}
We show that $\cA_k^1=\cA_k^2$.  Assume that $e_1\nparallel e_2$, so that $\lambda_H$ has the form
\eqref{eq:lambda-h} for some $\alpha_H>0$, $\beta_H>0$, and $\gamma_H\ne0$. We define $\tilde R_V$
satisfying \eqref{eq:RVe} as
\begin{equation}\label{eq:specialRV}
\tilde R_V = \frac{\alpha_H\xi_V\lambda_1-\beta_H\eta_V\lambda_2}{\alpha_H\lambda_1+\beta_H\lambda_2}
 = \frac{\lambda_{12}}{\lambda_H-\gamma_H},\quad e_1\nparallel e_2.
\end{equation}
Because $\alpha_H$ and $\beta_H$ are both positive, the denominator is not zero on $E$.  We
compute
\begin{align}\label{eq:RVinSpan}
\tilde R_V &= \frac{\lambda_{12}}{\lambda_H-\gamma_H}
=  - \frac{1}{\gamma_H}\lambda_{12} + \frac{1}{\gamma_H}\frac{\lambda_{12}\lambda_H}{\lambda_H-\gamma_H}
=  - \frac{1}{\gamma_H}\lambda_{12} + \frac{1}{\gamma_H}\lambda_H\tilde R_V.
\end{align}
We show that for any $\ell\geq0$,
\begin{equation}\label{eq:RV-inSpanInduction}
\tilde R_V = -\sum_{j=1}^{\ell}\frac{1}{\gamma_H^j}\lambda_{12}\lambda_H^{j-1}
 + \frac{1}{\gamma_H^{\ell}}\lambda_H^{\ell}\tilde R_V.
\end{equation}
The relation holds trivially for $\ell=0$, and \eqref{eq:RVinSpan} shows the result for $\ell=1$.  Assuming by induction that \eqref{eq:RV-inSpanInduction} holds for $\ell-1$, we compute (using \eqref{eq:RVinSpan})
\begin{align*}
\tilde R_V &= -\sum_{j=1}^{\ell-1}\frac{1}{\gamma_H^j}\lambda_{12}\lambda_H^{j-1}
 + \frac{1}{\gamma_H^{\ell-1}}\lambda_H^{\ell-1}\tilde R_V\\
&= -\sum_{j=1}^{\ell-1}\frac{1}{\gamma_H^j}\lambda_{12}\lambda_H^{j-1}
 + \frac{1}{\gamma_H^{\ell-1}}\lambda_H^{\ell-1}
               \bigg(-\frac{1}{\gamma_H}\lambda_{12}  + \frac{1}{\gamma_H}\lambda_H\tilde R_V\bigg)\\
&= -\sum_{j=1}^{\ell}\frac{1}{\gamma_H^j}\lambda_{12}\lambda_H^{j-1}
 + \frac{1}{\gamma_H^{\ell}}\lambda_H^{\ell}\tilde R_V,
\end{align*}
and \eqref{eq:RV-inSpanInduction} holds for all $\ell$.  Therefore, for any $0\leq m\leq k$, taking $\ell=k-m$,
\begin{equation*}
\lambda_H^m\tilde R_V = -\sum_{j=1}^{k-m}\frac{1}{\gamma_H^j}\lambda_{12}\lambda_H^{j-1+m}
 + \frac{1}{\gamma_H^{k-m}}\lambda_H^{k}\tilde R_V \in \cA_k^1,
\end{equation*}
and we conclude that $\cA_k^2\subset\cA_k^1$.  But clearly $\dim\cA_k^2=2(k-1)$ and
$\dim\cA_k^1\leq2(k-1)$, and so the spaces are in fact equal.

If $e_1\parallel e_2$, then $\lambda_{12}=\xi_V\lambda_1-\eta_V\lambda_2$ and
$\lambda_1+\lambda_2=\delta_H>0$ is a constant.  We define
\begin{equation}\label{eq:specialRVparallel}
\tilde R_V = \frac{\xi_V\lambda_1-\eta_V\lambda_2}{\lambda_1+\lambda_2}
 = \frac{1}{\delta_H}\lambda_{12},\quad e_1\parallel e_2,
\end{equation}
which satisfies \eqref{eq:RVe}. In this case, it is trivial that $\cA_k^2=\cA_k^1$.

By symmetry, $\cA_k^3=\cA_k^4$, where now
\begin{equation}\label{eq:specialRH}
\tilde R_H = \begin{cases}
\dfrac{\lambda_{34}}{\lambda_V-\gamma_V}, & e_3\nparallel e_4,\\
\dfrac{1\vphantom{H^H}}{\delta_V}\lambda_{34}, & e_3\parallel e_4,
\end{cases}
\end{equation}
where $\lambda_3+\lambda_4=\delta_V>0$ is a constant when $e_3\parallel e_4$.
\end{proof}

\begin{theorem}\label{thm:unisolvent}
  Let $(E,\cDS_r(E),\cN)$ be the $r$-th order direct serendipity finite element defined by
  \eqref{eq:shapeSpace}, i.e., by \eqref{eq:shape-vertices}, \eqref{eq:shape-cell}, and
  \eqref{eq:shape-h}--\eqref{eq:shape-vr}$)$ and \eqref{eq:nodalFunctionals}. If
  $\phi\in\cDS_r(E)$ and $N_k(\phi)=0$, for all $k= 1,2,\ldots,D_r$, then $\phi=0$.  Moreover,
\begin{equation}\label{eq:DSR}
\cDS_r(E)=\Po_r(E)\oplus\Supp_r^\cDS(E),
\end{equation}
where $\Supp_r^\cDS(E)$ is defined in \eqref{eq:supplementSpace}.
\end{theorem}

\begin{proof}
As noted in Subsection~\ref{sec:cell}, by construction, the full set of DoFs are unisolvent for
$\cDS_r(E)$ if the edge DoFs are unisolvent for the edge shape functions $\phi_{H,j}$ and
$\phi_{V,j}$.  We temporarily replace $R_V$ by $\tilde R_V$ defined in \eqref{eq:specialRV} or
\eqref{eq:specialRVparallel}.  In that case,
\begin{equation*}
\{\phi_{H,j} : j=0,1,\ldots,r-2\} = \lambda_3\lambda_4\cA_{r-2}^1
 = \lambda_3\lambda_4\cA_{r-2}^2,
\end{equation*}
by the lemma, and the representation of the space using $\cA_{r-2}^2$ clearly shows that the DoFs on
the edges $e_1$ and $e_2$ are unisolvent for $\phi_{H,j}$. That is, for an edge shape function
$\phi(\x)$ with vanishing DoFs, we can use $\cA_{r-2}^2$ to expand it as
\begin{equation*}
\phi(\x) = \lambda_3(\x)\lambda_4(\x)\sum_{\ell=0}^{r-2}\big(a_\ell + R_V(\x)\,b_\ell\big)\lambda_H^\ell(\x),
\end{equation*}
for some constants $a_\ell$ and $b_\ell$. On either edge $e_1$ or $e_2$ , $\phi(\x)$ is a polynomial of
degree $r$, which must vanish due to the vanishing of the DoFs.  Therefore,
$a_\ell + R_V(\x)\,b_\ell$ must vanish on each edge, and we conclude that both
$\eta_V a_\ell-\xi_V b_\ell=0$ and $\eta_V a_\ell+\xi_Vb_\ell=0$, i.e., $a_\ell=b_\ell=0$, and so
$\phi(\x)=0$.  It should be clear that we can return to the original $R_V$ and draw the same
conclusion, since the DoFs of any function in the argument are unchanged by this replacement. (Of
course, we no longer have that $\cA_{r-2}^1=\cA_{r-2}^2$, but only that their DoFs agree.) Similarly
we conclude unisolvence for $\phi_{V,j}$.

We conclude that $\dim\cDS_r(E)=D_r=\dim\Po_r(E)+2$. Since we clearly added only two shape functions
not in $\Po_r(E)$, the full space of polynomials is contained in $\cDS_r(E)$, and
$\cDS_r(E)=\Po_r(E)\oplus\Supp_r^\cDS(E)$.
\end{proof}


\subsection{Implementation as an $H^1$-Conforming Space}\label{sec:implementation}

The global direct serendipity finite element space of index $r\ge2$ over $\cT_h$ is
\begin{equation}
\cDS_r=\{v_h\in\cC^0(\Omega)\;:\; v_h|_E\in\cDS_r(E)\ \forall E\in \cT_h\}\subset H^1(\Omega).
\end{equation}
To implement $H^1$-conforming direct serendipity elements on the mesh $\cT_h$ over $\Omega$, we need
to find a proper basis for the finite element space, i.e., one that is continuous. We observe that
the interior cell shape functions, after extension by zero, are in $H^1(\Omega)$ and so cause no
difficulty.  The vertex and edge shape functions with DoFs on the boundaries of the elements must be
merged continuously. The simplest way to do this is to create a local nodal basis on every
$E\in\cT_h$ for these $4r$ shape functions.

Let $\phi$ be any one of the shape functions of $\cDS_r(E)$ defined above in \eqref{eq:shape-vertices},
\eqref{eq:shape-cell}, and \eqref{eq:shape-h}--\eqref{eq:shape-vr}.  To reduce rounding errors in
numerical calculations, we scale it so
\begin{align*}
\phi \quad\text{is replaced by}\quad \frac{\phi}{d_E^{n_\phi}},
\end{align*}
where $d_E = \sqrt{|E|}$ and $n_\phi$ is the degree of $\phi$ when it is a polynomial and
$n_{\phi_{H,2r-3}}=n_{\phi_{V,2r-3}}=r$.

In general, we can find the local basis $\{\varphi_1,\ldots,\varphi_{4r}\}$ by solving a small local
linear system. We order the shape functions with the vertex DoFs first ($\phi_1$ to $\phi_4$), the
horizontal edge DoFs on $e_1$ and $e_2$ next ($\phi_5$ to $\phi_{2r+2}$), and finally the vertical
edge DoFs on $e_3$ and $e_4$ ($\phi_{2r+3}$ to $\phi_{4r}$). We also order the DoFs similarly.
Construct the $4r\times 4r$ matrix ${\bf A}=(a_{ij})$ of the DoFs, i.e.,
${\bf A}_{ij} = N_j(\phi_i)$ for all $i,j\leq 4r$.  This matrix has a simple block structure based
on the DoFs on the vertices, edges $e_1$ and $e_2$, and edges $e_3$ and $e_4$, which is
\begin{align}
\bf{A} = \left(
\begin{matrix}
{\bf{A}}_{11} &  {\bf{A}}_{12}  & {\bf{A}}_{13}\\
{\bf0} & {\bf{A}}_{22} & {\bf0}\\
{\bf0} & {\bf{0}} & {\bf{A}}_{33}
\end{matrix}
\right),
\end{align}
where ${\bf{A}}_{11}$ is of size $4\times 4$ and ${\bf{A}}_{22}$ and ${\bf{A}}_{33}$ are of size
$2(r-1)\times 2(r-1)$.  From Theorem~\ref{thm:unisolvent} (unisolvence), we know that $\bf A$ is
invertible.  Let ${\bf{A}}^{-1}={\bf{B}}=(b_{ij})$ and define
\begin{equation*}
\varphi_i = \displaystyle\sum_{j=1}^{4r} b_{ij}\phi_j
\quad\implies\quad
N_k(\varphi_i) = \sum_{j=1}^{4r}b_{ij} N_k(\phi_j) 
=\sum_{j=1}^{4r}b_{ij} a_{jk} = \delta_{ik},
\end{equation*}
and we have our nodal basis.  Graphical depiction of the basis in special cases can be found in
\cite{Tao_2017_phd, Arbogast_Tao_2017_serendipity}.  Visually, there is nothing unusual about these
basis functions.

If we take the choice outlined in Lemma~\ref{lem:equ-space}, we can write down the nodal basis
explicitly, using the facts that $\cA_{r-2}^1=\cA_{r-2}^2$ and $\cA_{r-2}^3=\cA_{r-2}^4$.  The
edge basis functions become, for $j=1,\ldots,r-1$,
\begin{align}
\label{eq:ds-baseE1}
\varphi_{e_1,j}(\x)
&= \frac{\lambda_3(\x)\lambda_4(\x)}{\lambda_3(\x_{e_1,j})\lambda_4(\x_{e_1,j})}
      \frac{\xi_V - \tilde R_V(\x)}{\xi_V + \eta_V}
      \prod_{\shortstack{\scriptsize$k=1$\\\scriptsize$k\ne j$}}^{r-1}
            \frac{\lambda_H(\x) - \lambda_H(\x_{e_1,k})}{\lambda_H(\x_{e_1,j})-\lambda_H(\x_{e_1,k})},\\
\varphi_{e_2,j}(\x)
\label{eq:ds-baseE2}
&= \frac{\lambda_3(\x)\lambda_4(\x)}{\lambda_3(\x_{e_2,j})\lambda_4(\x_{e_2,j})}
      \frac{\tilde R_V(\x) + \eta_V}{\xi_V + \eta_V}
      \prod_{\shortstack{\scriptsize$k=1$\\\scriptsize$k\ne j$}}^{r-1}
            \frac{\lambda_H(\x) - \lambda_H(\x_{e_2,k})}{\lambda_H(\x_{e_2,j})-\lambda_H(\x_{e_2,k})},\\
\varphi_{e_3,j}(\x)
\label{eq:ds-baseE3}
&= \frac{\lambda_1(\x)\lambda_2(\x)}{\lambda_1(\x_{e_3,j})\lambda_2(\x_{e_3,j})}
      \frac{\xi_H - \tilde R_H(\x)}{\xi_H + \eta_H}
      \prod_{\shortstack{\scriptsize$k=1$\\\scriptsize$k\ne j$}}^{r-1}
             \frac{\lambda_V(\x) - \lambda_V(\x_{e_3,k})}{\lambda_V(\x_{e_3,j})-\lambda_V(\x_{e_3,k})},\\
\varphi_{e_4,j}(\x)
\label{eq:ds-baseE4}
&= \frac{\lambda_1(\x)\lambda_2(\x)}{\lambda_1(\x_{e_4,j})\lambda_2(\x_{e_4,j})}
      \frac{\tilde R_H(\x) + \eta_H}{\xi_H + \eta_H}
      \prod_{\shortstack{\scriptsize$k=1$\\\scriptsize$k\ne j$}}^{r-1}
             \frac{\lambda_V(\x) - \lambda_V(\x_{e_4,k})}{\lambda_V(\x_{e_4,j})-\lambda_V(\x_{e_4,k})}.
\end{align}
The vertex basis function $\varphi_{v,13}(\x)$ can be computed by first defining
\begin{equation*}
\tilde\phi_{v,13}(\x) = \frac{\lambda_2(\x)\lambda_4(\x)}{\lambda_2(\x_{v,13})\lambda_4(\x_{v,13})},
\end{equation*}
and then defining
\begin{align}
\label{eq:ds-baseV13}
\varphi_{v,13}(\x) &= \tilde\phi_{v,13}(\x) - \sum_{k\in\{1,3\}}\sum_{j=1}^{r-1} \tilde\phi_{v,13}(\x_{e_k,j})\,\varphi_{e_k,j}(\x).
\end{align}
The basis functions $\varphi_{v,14}(\x)$, $\varphi_{v,23}(\x)$, and $\varphi_{v,24}(\x)$ can be defined similarly.


\section{Serendipity supplements based on mapping from a reference element}\label{sec:mappingSupp}

The supplemental functions $\Supp_r^\cDS(E)$ used in the definition of $\cDS_r(E)$ in
\eqref{eq:shapeSpace} can be defined in terms of the bilinear map $\F_{\!E}:\hat E\to E$ and
$F_{\!E}^0$ discussed in Section~\ref{sec:notation}.  For example, since $\hat E=[-1,1]^2$, one can
define
\begin{align}
\label{eq:rational-h-v-piola}
R_V(\x) = F_{\!E}^0(\hat x_1)
\quad\text{and}\quad
R_H(\x) = F_{\!E}^0(\hat x_2),
\end{align}
for which $\eta_V=\xi_V=\eta_H=\xi_H=1$.

We can also use the map to define the entire supplemental functions \eqref{eq:rational-h} and
\eqref{eq:rational-v} themselves. For example, we can substitute the definitions
\begin{align}
\label{eq:shape-hr-map}
\phi_{H,2r-3}(\x) &= F_{\!E}^0\big((1-\hat x_2^2)\hat x_1\hat x_2^{r-2}\big),\\
\label{eq:shape-vr-map}
\phi_{V,2r-3}(\x) &= F_{\!E}^0\big((1-\hat x_1^2)\hat x_2\hat x_1^{r-2}\big),
\end{align}
giving direct serendipity elements with mapped supplements.  We must show unisolvence with this
substitution.  We proceed to show this property for edges $e_1$ and $e_2$. The other two edges will
then have this property by symmetry.

Easily,
\begin{equation*}
\phi_{H,2r-3}(\x) = F_{\!E}^0(1-\hat x_2^2)\,F_{\!E}^0(\hat x_1)\,\big(F_{\!E}^0(\hat x_2)\big)^{r-2}
 = F_{\!E}^0(1-\hat x_2^2)\,R_V\,(\lambda_H^*)^{r-2},
\end{equation*}
where $R_V$ is defined in \eqref{eq:rational-h-v-piola} and 
\begin{equation*}
\lambda_H^*=F_{\!E}^0(\hat x_2)
\end{equation*}
is a \emph{nonlinear} function.  Because $\F_{\!E}$ is a bilinear map, on the edges $e_1$ and
$e_2$, $\lambda_H^*$ is linear and $F_{\!E}^0(1-\hat x_2^2)$ is quadratic.  However, these may be
\emph{different} linear and quadratic functions on each edge.

The function $\lambda_H^*=F_{\!E}^0(\hat x_2)$ has the zero set being the line joining the center
of $e_1$ to the center of $e_2$.  Let $\x_H$ be any point on this line and $\nu_H$ denote a unit
normal to the line.  Define
\begin{equation}\label{eq:lambdaHnuH}
\lambda_H(\x) = -(\x - \x_H)\cdot\nu_H.
\end{equation}
If $e_1$ and $e_2$ are not parallel, then there exist (up to sign, so without loss of generality)
$\alpha_H>0$, $\beta_H>0$, and $\gamma_H\ne0$ such that
\begin{equation*}
\lambda_H(\x) = \alpha_H\lambda_1(\x) + \beta_H\lambda_2(\x) + \gamma_H.
\end{equation*}
If $e_1$ and $e_2$ are parallel, set $\alpha_H=\beta_H=1$ and $\gamma_H=0$ to obtain the same
representation of $\lambda_H$.  In either case, we define
\begin{equation*}
\lambda_{12} = \alpha_H\lambda_1 - \beta_H\lambda_2
\quad\text{and}\quad
\tilde R_V=\frac{\alpha_H\lambda_1 - \beta_H\lambda_2}{\alpha_H\lambda_1 + \beta_H\lambda_2}.
\end{equation*}
These functions satisfy the requirements of Lemma~\ref{lem:equ-space}.

Because $\lambda_H^*$ is linear on $e_1$ and $e_2$, there are nonzero constants $a$ and
$b$ of the same sign such that
\begin{equation}
\lambda_H^*\big|_{e_1} = a\lambda_H\big|_{e_1}
\quad\text{and}\quad
\lambda_H^*\big|_{e_2} = b\lambda_H\big|_{e_2}.
\end{equation}
Therefore, on the sides $e_1\cup e_2$,
\begin{align}\label{eq:piolaDS-edgeFcn-equ}
&\lambda_H(\x)^{r-2}\tilde R_V(\x)\\
\nonumber
&\quad= \frac{a^{r-2} - b^{r-2}}{a^{r-2} + b^{r-2}}\lambda_H(\x)^{r-2}
     + \frac{2}{a^{r-2} + b^{r-2}}(\lambda_H(\x)^*)^{r-2}\tilde R_V(\x),
\quad\x\in e_1\cup e_2.
\end{align}

We define the sets
\begin{align*}
\cA^1 &= {\rm{span}}\{ 1,\lambda_H,\ldots,\lambda_H^{r-2}, 
                      \lambda_{12},\lambda_H\lambda_{12},
                      \ldots, \lambda_H^{r-3}\lambda_{12},\lambda_H^{r-2}\tilde R_V\},\\
\cA^2 &= {\rm{span}}\big\{\{1,\lambda_H,\ldots,\lambda_H^{r-2}\}
                               \otimes\{1,\tilde R_V\}\big\},\\
\cA^{1,\sim} &= {\rm{span}}\{ 1,\lambda_H,\ldots,\lambda_H^{r-2},
                      \lambda_{12},\lambda_H\lambda_{12},
                      \ldots, \lambda_H^{r-3}\lambda_{12},(\lambda_H^{*})^{r-2}\tilde R_V\},\\
\cA^{1,*} &= {\rm{span}}\{ 1,\lambda_H,\ldots,\lambda_H^{r-2},
                      \lambda_{12},\lambda_H\lambda_{12},
                      \ldots, \lambda_H^{r-3}\lambda_{12},(\lambda_H^{*})^{r-2}R_V\},
\end{align*}
The later set $\cA^{1,*}$, times $F_{\!E}^0(1-\hat x_2^2)$, defines the shape functions for our
direct serendipity finite element based on the mapped supplement \eqref{eq:shape-hr-map}.
However, we can replace $R_V$ by $\tilde R_V$, since we consider only DoFs. That is, $\cA^{1,\sim}$
and $\cA^{1,*}$ are equivalent for our purposes.  Lemma~\ref{lem:equ-space} shows that
$\cA^1=\cA^2$, which is unisolvent.  Moreover, \eqref{eq:piolaDS-edgeFcn-equ} shows that
$\cA^1\subset\cA^{1,\sim}$, which have the same dimension and so are equal, and hence we have
unisolvence.  Unisolvence is maintained after multiplication by $F_{\!E}^0(1-\hat x_2^2)$, since
this modification concerns the fact that there are $c_1>0$ and $c_2>0$ so that
\begin{equation}\label{eq:mappedBubblePart}
F_{\!E}^0(1-\hat x_2^2)\big|_{e_j}=c_j\lambda_3\lambda_4,\quad j=1,2.
\end{equation}
We conclude that the direct serendipity element with the mapped supplements
\eqref{eq:shape-hr-map}--\eqref{eq:shape-vr-map}, i.e.,
\begin{equation}\label{eq:DSmap}
\cDS_r^{\text{map}}(E)
 = \Po_r(E)\oplus\spn\{F_{\!E}^0\big((1-\hat x_2^2)\hat x_1\hat x_2^{r-2}\big),
                                  F_{\!E}^0\big((1-\hat x_1^2)\hat x_2\hat x_1^{r-2}\big)\},
\end{equation}
is well defined.


\section{The de Rham complex and mixed finite elements}\label{sec:deRham}

The de Rham complex of interest here is
\begin{equation}\label{eq:deRham}
\Re \hooklongrightarrow H^1 \overset{\Curl\,}{\longlongrightarrow}
H(\Div) \overset{\Div\,}{\longlongrightarrow} L^2 \longrightarrow 0,
\end{equation}
where the curl (or rot) of a scalar function $\phi(\x)=\phi(x_1,x_2)$ is
$\Curl\,\phi = \bigg(\dfrac{\d\phi}{\d x_2},-\dfrac{\d\phi}{\d x_1}\bigg)$.  From right to left, the
image of one linear map is the kernel of the next. On rectangular elements, it is known
\cite{Arnold_Awanou_2011,Arnold_Awanou_2014} that the serendipity space $\cS_{r+1}$ is the precursor
of the Brezzi-Douglas-Marini space BDM$_r$ \cite{BDM_1985} for $r\geq1$; that is, on the reference
square $\hat E$, \eqref{eq:deRhamBDM} holds.


\subsection{Full and reduced AC spaces}\label{sec:ACspaces}

We have the following extension of \eqref{eq:deRhamBDM} to quadrilateral elements $E$. The direct
serendipity spaces $\cDS_r^{\text{map}}$ using the mapped supplements
\eqref{eq:shape-hr-map}--\eqref{eq:shape-vr-map} is the precursor of the reduced
$H(\Div)$-approximating Arbogast-Correa space AC$_r^{\textrm{red}}$ \cite{Arbogast_Correa_2016},
$r\geq1$, defined on meshes of convex quadrilaterals:
\begin{equation}\label{eq:deRhamACred}
\Re \hooklongrightarrow \cDS_{r+1}^{\text{map}}(E) \overset{\Curl\,}{\longlongrightarrow}
 \textrm{AC}_r^{\textrm{red}}(E) \overset{\Div\,}{\longlongrightarrow} \Po_{r-1}(E) \longrightarrow 0.
\end{equation}
Moreover, the full $H(\Div)$-approximating space AC$_r$, for $r\geq1$, satisfies
\begin{equation}\label{eq:deRhamAC}
\Re \hooklongrightarrow \cDS_{r+1}^{\text{map}}(E) \overset{\Curl\,}{\longlongrightarrow}
 \textrm{AC}_r(E) \overset{\Div\,}{\longlongrightarrow} \Po_{r}(E) \longrightarrow 0.
\end{equation}

This observation is clear once one realizes three sets of facts.  First, the direct serendipity
elements based on \eqref{eq:shape-hr-map}--\eqref{eq:shape-vr-map} have the structure
\begin{align}
\cDS_{r+1}^{\text{map}}(E) &= \Po_{r+1}(E)\oplus\Supp_{r+1}^{\cDS,\text{map}}(E),\\
\Supp_{r+1}^{\cDS,\text{map}}(E)
 &= \spn\big\{F_{\!E}^0\big((1-\hat x_2^2)\hat x_1\hat x_2^{r-1}\big),
                      F_{\!E}^0\big((1-\hat x_1^2)\hat x_2\hat x_1^{r-1}\big)\big\}.
\end{align}
Second, the AC elements have the structure
\begin{align}
\label{eq:AC}
\textrm{AC}_r(E) &= \textrm{AC}_r^{\textrm{red}}(E)\oplus\x\tilde\Po_{r}(E),\\
\label{eq:ACred}
\textrm{AC}_r^{\textrm{red}}(E) &= \Po_r^2(E)\oplus\Supp_{r+1}^{\textrm{AC}}(E),\\
\label{eq:ACsupp}
\Supp_{r+1}^{\textrm{AC}}(E)
&= \spn\big\{\F_{\!E}^1\,\Curl\big((1-\hat x_2^2)\hat x_1\hat x_2^{r-1}\big),
                     \F_{\!E}^1\,\Curl\big((1-\hat x_1^2)\hat x_2\hat x_1^{r-1}\big)\big\},
\end{align}
where $\F_{\!E}^1$ is the Piola mapping from $E$ to $\hat E$.  Finally, we have the fairly well-known
helmholtz-like decomposition (see, e.g., \cite{Arbogast_Correa_2016})
\begin{equation}
\Po_{r}^2(E) = \Curl\,\Po_{r+1}^2(E)\oplus\x\tilde\Po_{r-1}(E),
\end{equation}
the relation between the $\Curl$ operator and the bilinear and Piola maps
\begin{equation}
\Curl\,F_{\!E}^0 = \F_{\!E}^1\,\Curl,
\end{equation}
and the fact that the $\Div$ operator takes $\x\Po_{k}$ one-to-one and onto $\Po_{k}$ for any
$k\geq0$.

Now we see that
\begin{align}
\Curl\,\Supp_{r+1}^{\cDS,\text{map}}(E)
 &= \spn\big\{\Curl\,F_{\!E}^0\big((1-\hat x_2^2)\hat x_1\hat x_2^{r-1}\big),
                      \Curl\,F_{\!E}^0\big((1-\hat x_1^2)\hat x_2\hat x_1^{r-1}\big)\big\}\\
\nonumber
 &= \spn\big\{\F_{\!E}^1\,\Curl\big((1-\hat x_2^2)\hat x_1\hat x_2^{r-1}\big),
                      \F_{\!E}^1\,\Curl\big((1-\hat x_1^2)\hat x_2\hat x_1^{r-1}\big)\big\}\\
\nonumber
 &= \Supp_r^{\textrm{AC}}(E),
\end{align}
\vspace*{-20pt}\par\noindent
and so
\begin{equation}
\Curl\,\cDS_{r+1}^{\text{map}}(E) = \Curl\,\Po_{r+1}(E)\oplus\Supp_r^{\textrm{AC}}(E)
\end{equation}
is in the kernel of the operator $\Div$.  Finally, 
\begin{align}
\textrm{AC}_r^{\textrm{red}}(E) &= \Curl\,\Po_{r+1}(E)\oplus\Supp_r^{\textrm{AC}}(E)\oplus\x\tilde\Po_{r-1}(E),\\
\textrm{AC}_r(E) &= \Curl\,\Po_{r+1}(E)\oplus\Supp_r^{\textrm{AC}}(E)\oplus\x\tilde\Po_{r}(E),
\end{align}
satisfy the properties of the de Rham complex \eqref{eq:deRhamACred}--\eqref{eq:deRhamAC}.

We remark that it is easy to check that $\cS_1(E)$ (see \eqref{eq:S1}) precedes the element AC$_0$(E)
in the de Rham sequence \eqref{eq:deRhamAC}.


\subsection{Direct mixed finite elements on quadrilaterals}\label{sec:newMixed}

The de Rham theory provides a way of constructing a mixed finite element space $\V_r$ based on a well
defined direct serendipity space. Tangential derivatives of functions in $\cDS_{r+1}(E)$ along the
edges map by the $\Curl$ operator to normal derivatives; that is, if we define the unit tangential
vector
\begin{equation}
\label{eq:tau}
\tau_i=(-\nu_{i,2},\nu_{i,1})\quad\text{on }e_i,
\end{equation}
then for $\phi\in\cDS_{r+1}(E)$,
\begin{equation}\label{eq:tangentialToNormalDer}
\grad\phi\cdot\tau_i\big|_{e_i} = \Curl\,\phi\cdot\nu_i\big|_{e_i}.
\end{equation}
Since $\Curl\,\cDS_{r+1}(E)$ spans $\Po_{r}(e_i)$ independently of the other sides, the same is true of the
normal derivatives of $\V_r(E)$. In fact, for $r\geq1$, we have de Rham complexes for both full and
reduced direct $H(\Div)$-approximating mixed elements:
\begin{align}
\label{eq:deRhamV}
\Re \hooklongrightarrow \cDS_{r+1}(E) \overset{\Curl\,}{\longlongrightarrow}
 &\V_r^{\textrm{full}}(E) \overset{\Div\,}{\longlongrightarrow} \Po_{r}(E) \longrightarrow 0,\\
\label{eq:deRhamVred}
\Re \hooklongrightarrow \cDS_{r+1}(E) \overset{\Curl\,}{\longlongrightarrow}
 &\V_r^{\textrm{red}}(E) \overset{\Div\,}{\longlongrightarrow} \Po_{r-1}(E) \longrightarrow 0,
\end{align}
for any variant of our new direct serendipity spaces.  To see this fact, we need to decompose
$\V_r(E)$ (i.e, $\V_r^{\text{full}}(E)$ or $\V_r^{\text{red}}(E)$).

According to \cite{Arbogast_Correa_2016}, a reduced or full $H(\Div)$-approximating mixed finite
element space defined directly on a quadrilateral $E$ of minimal local dimension takes the form
($\cP$ in Definition~\ref{defn:ciarlet})
\begin{alignat}2
\label{eq:generalAC}
\V_r^{\text{full}}(E) &= \Po_r^2(E)\oplus\x\tilde\Po_{r}\oplus\Supp_r^\V(E)
 &&= \Curl\,\Po_{r+1}(E)\oplus\x\Po_{r}\oplus\Supp_r^\V(E),\\
\label{eq:generalACred}
\V_r^{\text{red}}(E) &= \Po_r^2(E)\oplus\Supp_r^\V(E)
 &&= \Curl\,\Po_{r+1}(E)\oplus\x\Po_{r-1}\oplus\Supp_r^\V(E),
\end{alignat}
where the choice of $\Supp_r^\V(E)$ is given by taking \eqref{eq:ACsupp}.  However, it is noted that
other supplemental functions could be used \cite[near (3.15)]{Arbogast_Correa_2016}.  Their normal
components must lie in $\Po_r(e_i)$ on each edge $e_i$ and, if they are mapped by the Piola
transform, they must contain a nontrivial component of the DoFs of $\Curl\,\hat x^{r+1}\hat y$ and
$\Curl\,\hat x\hat y^{r+1}$.

As given in \cite{Arbogast_Correa_2016}, the DoFs ($\cN$ in Definition~\ref{defn:ciarlet}) for
$\bfpsi\in\V_r^{\text{full}}(E)$ ($s=r$) or $\bfpsi\in\V_r^{\text{red}}(E)$ ($s=r-1$) are
\begin{alignat}2
\label{eq:vdofEdge}
&\int_{e_i}\bfpsi\cdot\nu_i\,p\,dx,&&\quad\forall p\in\Po_r(e_i),\ i=1,2,3,4,\\
\label{eq:vdofDiv}
&\int_{E}\bfpsi\cdot\grad q\,dx,&&\quad\forall q\in\Po_{s}(E),\\
\label{eq:vdofCell}
&\int_{E}\bfpsi\cdot\v\,dx,&&\quad\forall p\in\Bu_r^\V(E),
\end{alignat}
where the $H(\Div)$ bubble functions are
\begin{equation}
\label{eq:vBubbles}
\Bu_r^\V(E) = \Curl\big(\lambda_1\lambda_2\lambda_3\lambda_4\Po_{r-3}(E)\big).
\end{equation}
The DoFs \eqref{eq:vdofDiv} are determined entirely by the part of $\V_r^{\text{full}}(E)$ or
$\V_r^{\text{red}}(E)$ in the decomposition \eqref{eq:generalAC}--\eqref{eq:generalACred} that is
$\x\Po_{r}$ or $\x\Po_{r-1}$, respectively.  The DoFs \eqref{eq:vdofCell} correspond to the interior
cell direct serendipity DoFs.  In fact, $\Bu_r^\V(E)$ is exactly the $\Curl$ of the span of the cell
shape functions \eqref{eq:shape-cell} for $\cDS_{r+1}(E)$.  The DoFs \eqref{eq:vdofEdge} correspond
to the edge and vertex DoFs of $\cDS_{r+1}(E)$.

We can use any of our direct serendipity spaces to define the supplemental space $\Supp_r^\V(E)$
needed by $\V_r^{\text{full}}(E)$ or $\V_r^{\text{red}}(E)$.  The rest of the space is composed of
polynomials, and so need not be defined by $\cDS_{r+1}(E)$ through the $\Curl$ operator, although
this strategy could be used to help construct a basis for $\V_r(E)$ respecting the DoFs.  On $E$
when $r\geq1$, our new mixed spaces use the supplemental space
\begin{equation}
\Supp_{r}^\V(E)=\Curl\,\Supp_{r+1}^\cDS(E).
\end{equation}

In particular, as we saw, $\Curl\,\Supp_{r+1}^{\cDS,\text{map}}(E)$ gives the supplements for the
elements AC$_r(E)$.  If we use the fully direct serendipity supplements from
\eqref{eq:supplementSpace}, we obtain new families of fully direct mixed elements.  The computations
are not difficult.  Note that
\begin{equation}
\Curl\,\lambda_j = -\Curl\big((\x-\x_j)\cdot\nu_j\big) = \tau_j.
\end{equation}
Suppose that $\lambda_H$ is represented by \eqref{eq:lambdaHnuH} (i.e., the zero line is orthogonal
to $\nu_H$) and define $\tau_H=(-\nu_{H,2},\nu_{H,1})$, so $\Curl\,\lambda_H=\tau_H$.  If we use
\eqref{eq:rational-v}--\eqref{eq:rational-h} to define $R_V$ and $R_H$, the supplemental space is
(recall \eqref{eq:supplementSpace})
\begin{align}
&\Supp_{r}^\V(E) = \spn\{\bfsigma_{r,1},\bfsigma_{r,2}\},\\
&\bfsigma_{r,1} = \Curl(R_V^{\text{simple}}\lambda_H^{r-1}\lambda_3\lambda_4)
= \lambda_H^{r-1}\lambda_3\lambda_4
 \frac{(\xi_V^{-1}+\eta_V^{-1})(\lambda_2\tau_1-\lambda_1\tau_2)}{(\xi_V^{-1}\lambda_1+\eta_V^{-1}\lambda_2)^2}\\
\nonumber
&\qquad\quad + (r-1)R_V^{\text{simple}}\lambda_H^{r-2}\lambda_3\lambda_4\tau_H
            + R_V^{\text{simple}}\lambda_H^{r-1}(\lambda_4\tau_3+\lambda_3\tau_4),\\
&\bfsigma_{r,2} = \Curl(R_H^{\text{simple}}\lambda_V^{r-1}\lambda_1\lambda_2)
= \lambda_V^{r-1}\lambda_1\lambda_2
 \frac{(\xi_H^{-1}+\eta_H^{-1})(\lambda_4\tau_3-\lambda_3\tau_4)}{(\xi_H^{-1}\lambda_3+\eta_H^{-1}\lambda_4)^2}\\
\nonumber
&\qquad\quad + (r-1)R_H^{\text{simple}}\lambda_V^{r-2}\lambda_1\lambda_2\tau_V
            + R_H^{\text{simple}}\lambda_V^{r-1}(\lambda_2\tau_1+\lambda_1\tau_2).
\end{align}
The normal flux on each edge is easy to compute.  For $\bfsigma_{r,1}$, we have
\begin{align*}
\bfsigma_{r,1}\cdot\nu_1\big|_{e_1} &=
  -\eta_V\lambda_H^{r-2}\big((r-1)\lambda_3\lambda_4\tau_H\cdot\nu_1 +
                                      \lambda_H\,(\lambda_4\tau_3\cdot\nu_1+\lambda_3\tau_4\cdot\nu_1)\big),\\
\bfsigma_{r,1}\cdot\nu_2\big|_{e_2} &= \xi_V\lambda_H^{r-2}\big((r-1)\lambda_3\lambda_4\tau_H\cdot\nu_2 +
                                      \lambda_H\,(\lambda_4\tau_3\cdot\nu_2+\lambda_3\tau_4\cdot\nu_2)\big),\\
\bfsigma_{r,1}\cdot\nu_3\big|_{e_3} &= \bfsigma_{r,1}\cdot\nu_4\big|_{e_4} = 0.
\end{align*}
It is readily apparent that, indeed, the normal fluxes are in $\Po_r(e_i)$ for each $i$.


\subsection{Implementation as an $H(\Div)$-conforming mixed space}\label{sec:implementationMixed}

The mixed space of vector functions $\V_r$ over $\Omega$ is defined by merging continuously the
normal fluxes across each edge $e$ of the mesh $\cT_h$. That is,
\begin{align}
\V_r^{\text{full}} &= \big\{\v\in H(\Div;\Omega)\;:\;\v\big|_E\in\V_r^{\text{full}}(E)\text{ for all }E\in\cT_h\big\},\\
\V_r^{\text{red}} &= \big\{\v\in H(\Div;\Omega)\;:\;\v\big|_E\in\V_r^{\text{red}}(E)\text{ for all }E\in\cT_h\big\}.
\end{align}
This can be done locally by constructing a local basis respecting the (edge) DoFs, in a way similar
to that described for the serendipity elements in Section~\ref{sec:implementation}.  However, in
practical implementation, the hybrid form of the mixed method is often used
\cite{Arnold_Brezzi_1985}.  In that case, the elements are simply concatenated and no DoF-basis is
required.  The Lagrange multiplier space, used to enforce the normal flux continuity, is simply
\begin{equation}
\Lambda_r = \big\{\lambda\in L^2\big(\cup_{E\in\cT_h}\d E\big)\;:\;
                          \lambda\big|_e\in\Po_r(e)\text{ for each edge }e\text{ of }\cT_f\big\}.
\end{equation}

The mixed space of vector functions $\V_r^{\text {full}}$ or $\V_r^{\text{red}}$ is normally paired
with a space approximating scalar functions
\begin{equation}
W_s = \big\{w\in L^2(\Omega)\;:\;w\big|_E\in\Po_s(E)\text{ for all }E\in\cT_h\big\},
\end{equation}
denoted $W_r^{\text {full}}=W_r$ or $W_r^{\text{red}}=W_{r-1}$, respectively. These spaces are the
divergences of the corresponding vector function spaces.


\section{Stability and convergence properties}\label{sec:properties}

In this section, we summarize the stability and convergence theory for our new direct finite
elements.  For the most part, we work over the entire domain $\Omega$.


\subsection{Direct serendipity element properties}\label{sec:serendipityProperties}

In Section~\ref{sec:implementation}, we discussed creating a local nodal basis for some of the shape
functions of $\cDS_r(E)$.  By Theorem~\ref{thm:unisolvent}, there exists a fully nodal basis; that
is, one for which every basis function vanishes at all but one nodal point. We denote it as
$\{\varphi_1,\ldots,\varphi_{\dim\cDS_r(E)}\}$.

\begin{definition}
  Given the $r$-th order direct serendipity element $(E, \cDS_r(E), \cN)$ and the nodal basis of
  $\cDS_r(E)$, $\{\varphi_1,\ldots,\varphi_{\dim\cDS_r(E)}\}$, let the operator
  $\cI_E:L^2(E)\cap\cC^0(E)\longrightarrow \cDS_r(E)$ be interpolation. That is, for
  $\phi\in L^2(E)\cap \cC^0(E)$,
\[
\cI_E\,\phi = \sum_{j=1}^{\dim \cDS_r(E)} N_j(\phi)\,\varphi_j
 = \sum_{j=1}^{\dim \cDS_r(E)} \phi(\x_j)\,\varphi_j\in\cDS_r(E).
\]
  Given the finite element space $\cDS_r$ over $\Omega$, let the operator $\cI_h$ be
  global interpolation. That is, for a given function $\phi\in L^2(\Omega)\cap \cC^0(\Omega)$,
  $\cI_h\,v\in\cDS_r$ and $\cI_h\,\phi\big|_E = \cI_E\,\phi$.
\end{definition}

By Theorem~\ref{thm:unisolvent}, the local interpolation operator preserves polynomials, so we have
an important property~\cite[pp.~121--123]{Ciarlet_1978} expressed in the following lemma.

\begin{lemma}
\label{thm:poly-preserve}
The interpolation operator $\cI_E$ is polynomial preserving, i.e., $\forall \psi \in \Po_r(E)$,
$\cI_E\,\psi = \psi$. Moreover, $\|\cI_E\|$ is bounded in the $L^2$-norm.
\end{lemma}

With this lemma and Theorem~\ref{thm:unisolvent}, we have the analogue of the
Bramble-Hilbert~\cite{Bramble_Hilbert_1970} or Dupont-Scott~\cite{Dupont_Scott_1980} lemma for
local and global error estimation, provided the mesh is shape regular.

\begin{lemma}\label{lem:bramble}
There exists a constant $C>0$ such that for all functions $\phi\in H^{s+1}(E)$
 $(H^1(E)\cap\cC^0(E)$ if $s=0)$,
\begin{align}
|\phi-\cI_E\,\phi|_{m,E}\leq C\,h_E^{s+1-m}\,|\phi|_{s+1,E}, \quad m=0,1\text{ and }s=0,1,\ldots,r,
\end{align}
where $|\cdot|_{m,E}$ is the $H^m(E)$ seminorm.  Moreover, suppose that $\cT_h$ is uniformly shape
regular as $h\to0$. Then there exists a constant $C>0$, independent of $h$, such that for all functions
$\phi\in H^{s+1}(\Omega)$ $(\phi\in H^1(\Omega)\cap \cC^0(\overline\Omega)$ if $s=0)$,
\begin{align}
|\phi - \cI_h\,\phi|_{m,\Omega}\leq C\,h^{s+1-m}\,|\phi|_{s+1,\Omega}, \quad m=0,1\text{ and }s=0,1\ldots,r.
\end{align}
\end{lemma}


\subsection{Direct mixed finite element properties}\label{sec:mixedProperties}

As was done by Raviart and Thomas \cite{Raviart_Thomas_1977} for their mixed spaces, we can define a
projection operator, usually denoted $\pi$, mapping $H(\Div;\Omega)\cap(L^{2+\epsilon}(\Omega))^2$,
$\epsilon>0$, onto $\V_r$ (full or reduced) that has several important properties.  The operator
$\pi$ is pieced together from locally defined operators $\pi_E$. Following
\cite{Arbogast_Correa_2016}, for $\v$, we define $\pi_E\v$ in terms of the DoFs
\eqref{eq:vdofEdge}--\eqref{eq:vdofCell}. The operator $\pi$ satisfies the commuting
diagram property~\cite{Douglas_Roberts_1985}, which is to say that
\begin{equation}
\cP_{W_s}\div\v=\div\pi\v,
\end{equation}
where $\cP_{W_s}$ is the $L^2$-orthogonal projection operator onto $W_s=\div\V_r$ and $s=r$ for full
spaces and $s=r-1$ for reduced.  Since $\pi_E$ is bounded in, say, $H^1$, we obtain the following
approximation results \cite{Bramble_Hilbert_1970, Dupont_Scott_1980, Brezzi_Fortin_1991,
  Ciarlet_1978, Boffi_Brezzi_Fortin_2013, Arbogast_Correa_2016}.

\begin{lemma}\label{lem:brambleMixed}
Suppose that $\cT_h$ is uniformly shape regular as $h\to0$. Then there is a constant $C>0$,
independent of $h$, such that
\begin{alignat}2
\label{eq:approx_u}
\|\v-\pi\v\|_{0,\Omega} &\le C\,\|\v\|_{k,\Omega}\,h^{k},&&\quad k=1,\ldots,r+1,\\
\label{eq:approx_divu}
\|\div(\v-\pi\v)\|_{0,\Omega} &\le C\,\|\div\u\|_{k,\Omega}\,h^{k},&&\quad k=0,1,\ldots,s+1,\\
\label{eq:approx_p}
\|p-\cP_{W_s}p\|_{0,\Omega} &\le C\,\|p\|_{k,\Omega}\,h^{k},&&\quad k=0,1,\ldots,s+1,
\end{alignat}
where $s=r\ge1$ and $s=r-1\ge0$ for full and reduced $H(\Div)$-approximation, respectively, and
$\|\cdot\|_{m,\Omega}$ is the $H^m(\Omega)$ norm.  Moreover, the discrete inf-sup condition
\begin{equation}
\label{eq:inf-sup}
\sup_{\v_h\in\V_r}\frac{(w_h,\div\v_h)}{\|\v_h\|_{H(\text{\rm div})}}
 \ge \gamma\,\|w_h\|_{0,\Omega},\quad\forall w_h\in W_s,
\end{equation}
holds for some $\gamma>0$ independent of $h$.
\end{lemma}


\subsection{Application to second order elliptic equations}\label{sec:ellipticProperties}

Consider a uniformly elliptic problem with a homogeneous Dirichlet boundary condition
\begin{alignat}{2}
\label{eq:elliptic-pde}
-\div(\a \grad p) &= f &&\quad\text{in }\Omega,\\
\label{eq:elliptic-dirBC}
 p &= 0 &&\quad\text{on }\partial\Omega,
\end{alignat}
where the second order tensor $\a(\x)$ is uniformly positive definite and bounded, and
$f\in L^2(\Omega)$.  The boundary value problem can be written in the weak form: Find
$p\in H_0^1(\Omega)$ such that
\begin{align}
\label{eq:weak-bvp}
(\a\grad p,\grad q) = (f,q), \quad\forall q\in H_0^1(\Omega),
\end{align}
where $(\cdot,\cdot)$ is the $L^2(\Omega)$ inner product. Setting
\begin{equation}
\label{eq:elliptic-flux}
\u = -\a\grad p,
\end{equation}
we also have the mixed weak form: Find $\u\in H(\Div;\Omega)$ and $p\in L^2(\Omega)$ such that
\begin{alignat}3
\label{eq:mixed-darcy}
&(\a^{-1}\u,\v) - (p,\div\v) &&=0, &&\quad\forall\v\in H(\Div;\Omega),\\
\label{eq:mixed-conservation}
&(\div\u,w) &&= (f,w), &&\quad\forall w\in L^2(\Omega).
\end{alignat}

Define the global finite element space over $\cT_h$
\begin{equation*}\label{eq:x0h-in-h01}
X_{0,h}=\{v_h\in\cDS_r\;:\; v_h=0\text{ on }\d\Omega\}\subset H_0^1(\Omega).
\end{equation*}
We then obtain the Galerkin approximation: Find $p_h\in X_{0,h}$ such that
\begin{align}
\label{eq:approx-bvp}
(\a\grad p_h,q_h) = (f,q_h), \quad\forall q_h\in X_{0,h}.
\end{align}
Combining C\'ea's lemma \cite{Ciarlet_1978,Brenner_Scott_1994} and the global projection estimate
Lemma~\ref{lem:bramble}, we obtain an $H^1$-error estimate for the problem.  Since $\Omega$ is a
polygonal domain, $\partial\Omega$ is a Lipschitz boundary.  If we assume that $\Omega$ is also
convex, we have elliptic regularity of the solution~\cite[Theorem~4.3.1.4]{Grisvard_1985}, and the
Aubin-Nitsche duality principle \cite{Ciarlet_1978,Brenner_Scott_1994} gives an $L^2$-error
estimate.

\begin{theorem}\label{thm:convergence}
  Let $\Omega$ be a convex polygonal domain and let $\cT_h$ be uniformly shape regular.  There
  exists a constant $C>0$, independent of $h$, such that
\begin{align}
\|p-p_h\|_{m,\Omega}&\leq C\,h^{s+1-m}\,|p|_{s+1,\Omega},\quad s=0,1,\ldots,r,\quad m=0,1,
\end{align}
where $p_h$ satisfies \eqref{eq:approx-bvp}.
\end{theorem}

We also have the mixed full ($s=r$) and reduced ($s=r-1$) $H(\Div)$-approximation: Find
$(\u_h,p_h)\in\V_r\times W_s$ such that 
\begin{alignat}3
\label{eq:mixed-darcy-approx}
&(\a^{-1}\u_h,\v_h) - (p_h,\div\v_h) &&=0, &&\quad\forall\v_h\in\V_r,\\
\label{eq:mixed-conservation-approx}
&(\div\u_h,w_h) &&= (f,w_h), &&\quad\forall w_h\in W_s.
\end{alignat}

\begin{theorem}\label{thm:convergenceMixed}
Let $\Omega$ be a convex polygonal domain and let $\cT_h$ be uniformly shape regular.  There
exists a constant $C>0$, independent of $h$, such that
\begin{alignat}2
\|\u-\u_h\| &\le C\|\u\|_{k}h^{k},&&\quad k=1,\ldots,r+1,\\
\|\div(\u-\u_h)\| + \|p-p_h\| &\le C\|\u\|_{k}h^{k},&&\quad k=0,1,\ldots,s+1,
\end{alignat}
where $(\u_h,p_h)$ satisfies \eqref{eq:mixed-darcy-approx}--\eqref{eq:mixed-conservation-approx}.
\end{theorem}


\section{Numerical results}\label{sec:numerics}

In this section, we consider the test problem~\eqref{eq:elliptic-pde}--\eqref{eq:elliptic-dirBC}
defined on the unit square $\Omega = [0,1]^2$ with the coefficient $\a$ being the $2\times2$
identity matrix, i.e., we solve the Poisson equation.  The exact solution is
$u(x,y) = \sin(\pi x)\sin(\pi y)$ and the source term is $f(x,y) = 2\pi^2\sin(\pi x)\sin(\pi y)$.

Solutions are computed on three different sequences of meshes.  The first sequence, $\cT_h^1$, is a
uniform mesh of $n^2$ square elements (two sets of parallel edges per element). The second sequence,
$\cT_h^2$, is a mesh of $n^2$ trapezoids of base $h$ and one pair of parallel edges of size $0.75h$
and $1.25h$, as proposed in~\cite{ABF_2002}.  The third sequence, $\cT_h^3$, is chosen so as to have
no pair of edges parallel.  The first $4\times 4$ meshes for each sequence are shown in
Figure~\ref{fig:meshes}.  Finer meshes are constructed by repeating the same pattern over the
domain. Our computer program uses the deal.II library \cite{BDHHKKMTW_2016_dealII84}.

\begin{figure}[ht]
\centerline{
\parbox{.16\linewidth}{\includegraphics[width=\linewidth]{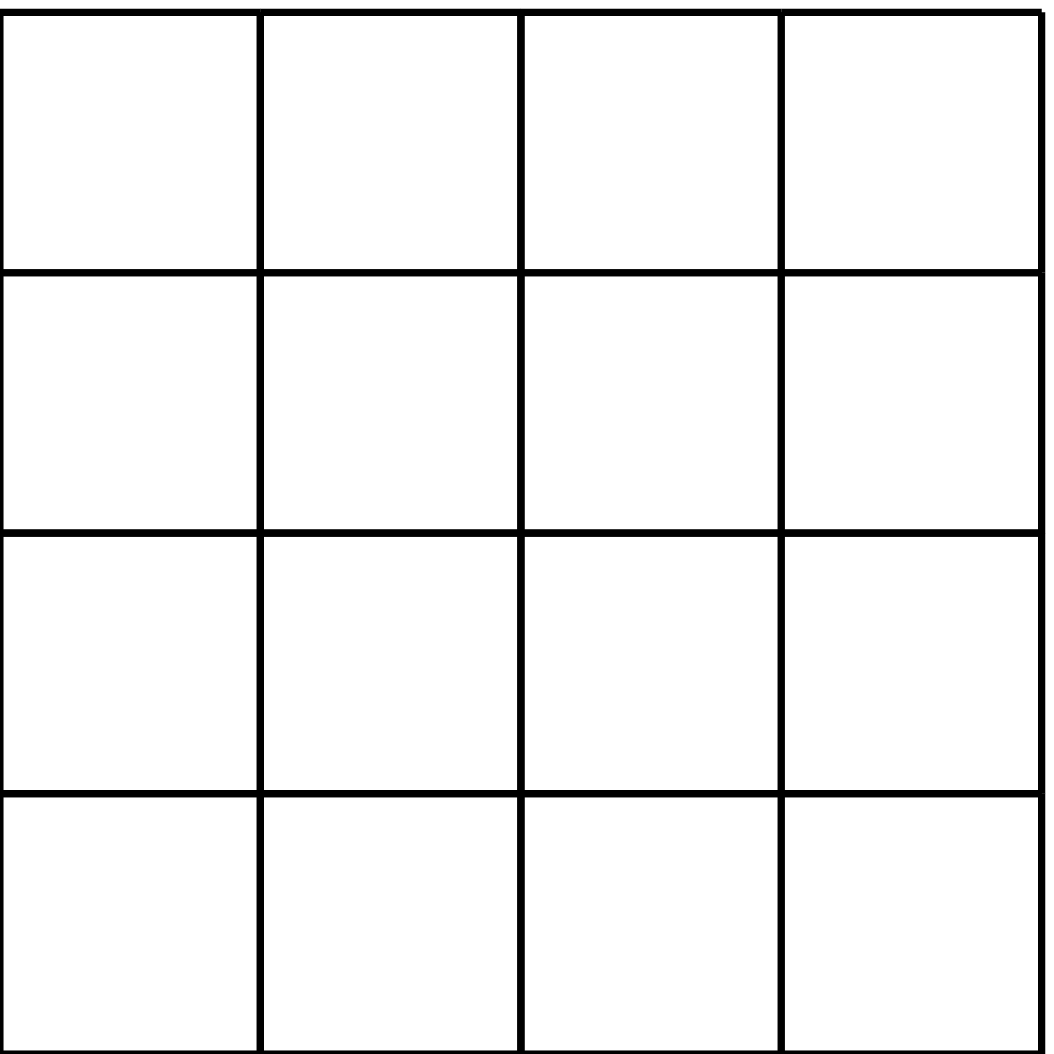}\newline\centerline{$\cT_h^1$}}
\qquad\qquad\parbox{.16\linewidth}{\includegraphics[width=\linewidth]{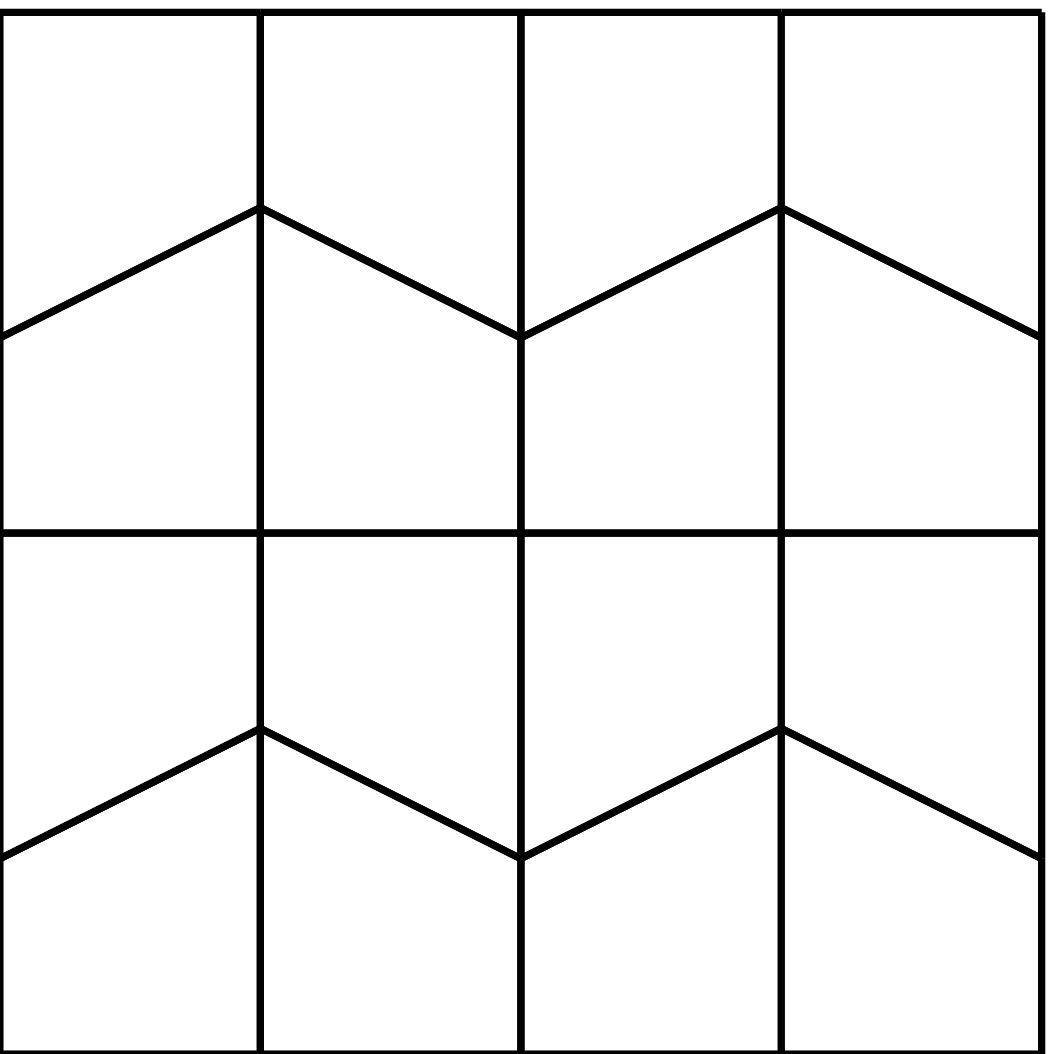}\newline\centerline{$\cT_h^2$}}
\qquad\qquad\parbox{.16\linewidth}{\includegraphics[width=\linewidth]{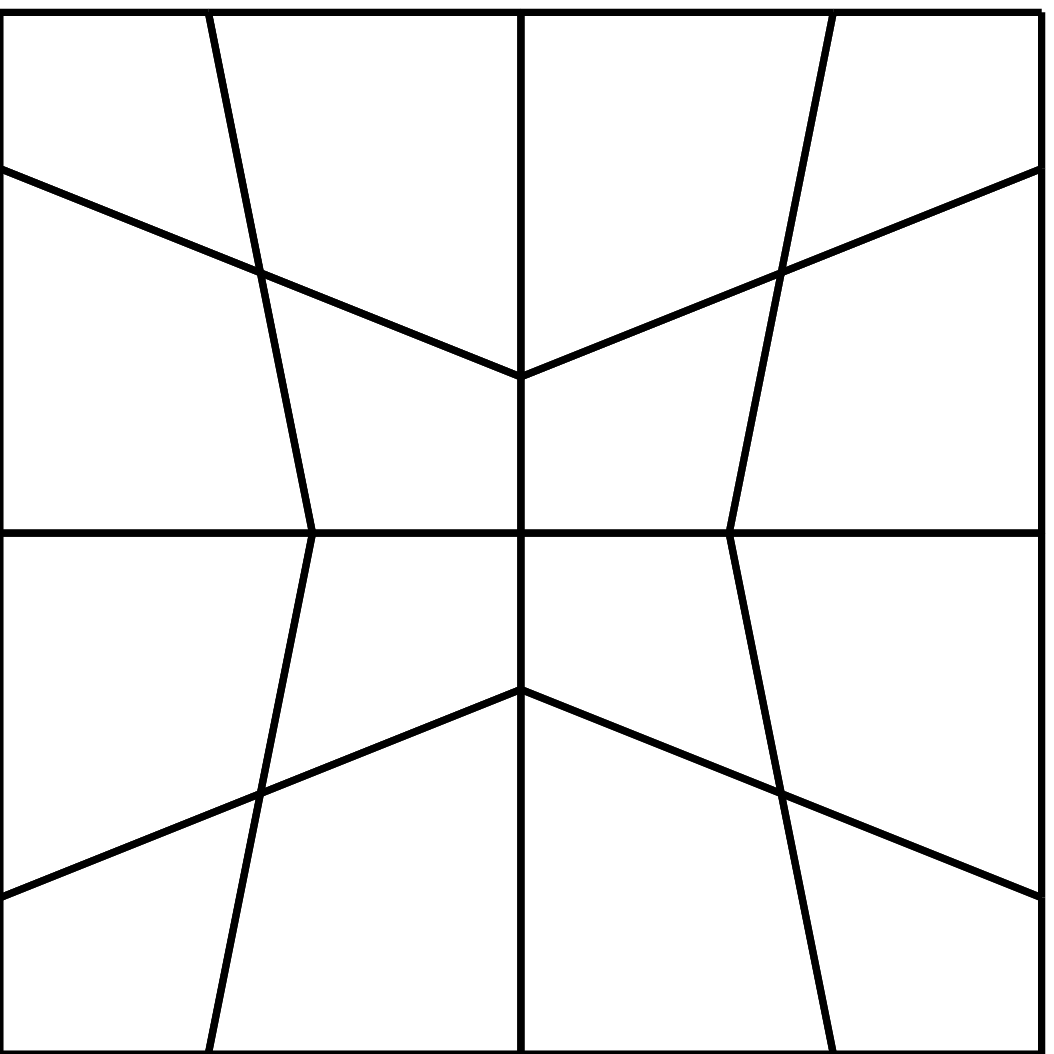}\newline\centerline{$\cT_h^3$}}
}
\caption{The three $4\times4$ base meshes. Finer meshes are constructed by repeating the base mesh
  pattern over the domain.  The meshes have 2, 1, and 0 parallel edges per element,
  respectively.}\label{fig:meshes}
\end{figure}


\subsection{Fully direct serendipity spaces}

In this section we present convergence studies for the fully direct serendipity spaces $\cDS_r$
using the elements defined in~\eqref{eq:shapeSpace}. We compare the results with those of the
regular serendipity spaces $\cS_r$ and the spaces of elements given by mapping the local tensor
product space $\Po_{r,r}(\hat E)$ to the mesh elements (hereafter simply called the $\Po_{r,r}$
space).

As described above, one may need to consider whether opposite faces are parallel to construct a
fully direct serendipity element. We now give a simple choice of element that avoids this
difficulty.  First, we can take \eqref{eq:simpleLambdaHV} for $\lambda_H$ and $\lambda_V$.  Let
$R_V$ and $R_H$ be defined by \eqref{eq:rational-v}--\eqref{eq:rational-h}, where we define
$\nu_H=(\nu_3-\nu_4)/|\nu_3-\nu_4|$ and $\nu_V=(\nu_1-\nu_2)/|\nu_1-\nu_2|$, and set
\begin{equation}\label{def:xi-eta-vh}
\begin{alignedat}2
\xi_V^{-1} &= \sqrt{1-(\nu_H\cdot\nu_1)^2},\qquad
&\eta_V^{-1} &= \sqrt{1-(\nu_H\cdot\nu_3)^2},\\
\xi_H^{-1} &= \sqrt{1-(\nu_V\cdot\nu_2)^2},\qquad
&\eta_H^{-1} &= \sqrt{1-(\nu_V\cdot\nu_4)^2}.
\end{alignedat}
\end{equation}

For an $n\times n$ mesh, the total number of degrees of freedom for $\Po_{r,r}$ is
$(nr+1)^2 = \cO(r^2n^2)$, and for $\cS_r$ and $\cDS_r$ it is
\begin{align*}
\dim(\cS_r) = \dim(\cDS_r) &= (\text{number of vertices}) + (\text{number of edges})(r-1)\\
&\qquad\quad + (\text{number of cells})\tfrac12{(r-2)(r-3)} \\
&=(n+1)^2 + 2n(n+1)(r-1) + n^2 \tfrac12{(r-2)(r-3)}\\
&=\tfrac12(r^2 - r + 4)n^2 + 2rn + 1 = \cO\left(\tfrac12{(r^2-r+4)n^2}\right).
\end{align*}
Therefore, the total number of degrees of freedom for a serendipity space is asymptotically about
half the size of that for a tensor product space of the same order.

We report the $L^2$-errors and the orders of the convergence of the spaces $\Po_{r,r}$, $\cS_r$ and
$\cDS_r$ for $r=2,3,4,5$ on mesh sequence $\cT_h^1$ in Table~\ref{tab:t1-l2}. The errors and
convergence rates in the $H^1$-seminorm are presented in Table~\ref{tab:t1-h1}.  Since $\cT_h^1$ is
a sequence of square meshes, the direct serendipity space $\cDS_r$ and the regular serendipity space
$\cS_r$ coincide on $\cT_h^1$.  All three families of spaces show an $(r+1)$-st order convergence in
the $L^2$-norm and an $r$-th order convergence in the $H^1$-seminorm, as we should expect from
theory.  The errors for $\Po_{r,r}$ are smaller than that for $\cDS_r=\cS_r$, but $\Po_{r,r}$ uses
many more degrees of freedom.

\begin{table}[!ht]
\caption{$L^2$-errors and convergence rates for
$\Po_{r,r}$, $\cDS_r$, and $\cS_r$ spaces on square meshes.}\label{tab:t1-l2}
\begin{center}
\begin{tabular}{c|cc|cc|cc|cc}
\hline
 &  \multicolumn{2}{c|}{$r=2$} & \multicolumn{2}{c|}{$r=3$}
 & \multicolumn{2}{c|}{$r=4$} & \multicolumn{2}{c}{$r=5$} \\
 $n$ & error & rate & error & rate & error & rate & error & rate \\
\hline
\multicolumn{9}{c}{$\myStrut \Po_{r,r}$ on $\cT_h^1$ meshes}\\
\hline
  \08 & 2.451e-04 & 2.99 & 5.564e-06 & 3.99 & 1.054e-07 & 4.99 & 1.688e-09 & 6.00 \\
   12 & 7.282e-05 & 2.99 & 1.101e-06 & 4.00 & 1.389e-08 & 5.00 & 1.483e-10 & 6.00 \\
   16 & 3.075e-05 & 3.00 & 3.486e-07 & 4.00 & 3.298e-09 & 5.00 & 2.640e-11 & 6.00 \\
   24 & 9.116e-06 & 3.00 & 6.890e-08 & 4.00 & 4.344e-10 & 5.00 & 2.420e-12 & 5.89 \\
\hline
\multicolumn{9}{c}{$\myStrut \cS_{r}=\cDS_{r}$ on $\cT_h^1$ meshes}\\
\hline
  \08 & 2.457e-04 & 2.99 & 1.805e-05 & 4.09 & 1.422e-06 & 5.01 & 6.440e-08 & 5.93 \\
   12 & 7.289e-05 & 3.00 & 3.497e-06 & 4.05 & 1.870e-07 & 5.00 & 5.739e-09 & 5.96 \\
   16 & 3.076e-05 & 3.00 & 1.099e-06 & 4.02 & 4.437e-08 & 5.00 & 1.027e-09 & 5.98 \\
   24 & 9.118e-06 & 3.00 & 2.161e-07 & 4.01 & 5.841e-09 & 5.00 & 9.049e-11 & 5.99 \\
\hline
\end{tabular}
\end{center}
\end{table}

\begin{table}[!ht]
\caption{$H^1$-seminorm errors and convergence rates for
$\Po_{r,r}$, $\cDS_r$, and $\cS_r$ spaces on square meshes.}\label{tab:t1-h1}
\begin{center}
\begin{tabular}{c|cc|cc|cc|cc}
\hline
 &  \multicolumn{2}{c|}{$r=2$} & \multicolumn{2}{c|}{$r=3$}
 & \multicolumn{2}{c|}{$r=4$} & \multicolumn{2}{c}{$r=5$} \\
 $n$ & error & rate & error & rate & error & rate & error & rate \\
 \hline
\multicolumn{9}{c}{$\myStrut \Po_{r,r}$ on $\cT_h^1$ meshes}\\
\hline
  \08 & 1.276e-02 & 2.00 & 4.233e-04 & 3.00 & 1.047e-05 & 4.00 & 2.066e-07 & 5.00 \\
   12 & 5.673e-03 & 2.00 & 1.255e-04 & 3.00 & 2.070e-06 & 4.00 & 2.723e-08 & 5.00 \\
   16 & 3.191e-03 & 2.00 & 5.295e-05 & 3.00 & 6.549e-07 & 4.00 & 6.462e-09 & 5.00 \\
   24 & 1.418e-03 & 2.00 & 1.569e-05 & 3.00 & 1.294e-07 & 4.00 & 8.511e-10 & 5.00 \\
\hline
\multicolumn{9}{c}{$\myStrut \cS_{r}=\cDS_{r}$ on $\cT_h^1$ meshes}\\
\hline
  \08 & 1.285e-02 & 2.02 & 1.537e-03 & 3.05 & 1.141e-04 & 3.99 & 5.201e-06 & 4.99 \\
   12 & 5.690e-03 & 2.01 & 4.507e-04 & 3.03 & 2.261e-05 & 3.99 & 6.856e-07 & 5.00 \\
   16 & 3.197e-03 & 2.00 & 1.894e-04 & 3.01 & 7.164e-06 & 4.00 & 1.628e-07 & 5.00 \\
   24 & 1.420e-03 & 2.00 & 5.597e-05 & 3.01 & 1.416e-06 & 4.00 & 2.144e-08 & 5.00 \\
\hline
\end{tabular}
\end{center}
\end{table}

Tables~\ref{tab:t2-l2}--\ref{tab:t2-h1} show the errors (in the $L^2$ and $H^1$-seminorms,
respectively) and the orders of convergence for the trapezoidal mesh sequence $\cT_h^2$.  The tensor
product space $\Po_{r,r}$ achieves the expected optimal convergence rates.  The direct serendipity
space $\cDS_r$ retains an optimal $(r+1)$-st order of convergence in the $L^2$ norm and an optimal
$r$-th order convergence in the $H^1$-seminorm, as Theorem~\ref{thm:convergence} predicts.  The
regular serendipity spaces $\cS_r$ have worse than optimal convergence rates in both norms (as was
also observed in~\cite{ABF_2002}).  The errors and convergence rates for $\Po_{r,r}$, $\cDS_r$, and
$\cS_r$ on mesh sequence $\cT_h^3$ are similar to those on $\cT_h^2$, so we omit showing them.

\begin{table}[ht]
\caption{$L^2$-errors and convergence rates for
$\Po_{r,r}$, $\cDS_r$, and $\cS_r$ spaces on trapezoidal meshes.}\label{tab:t2-l2}
\begin{center}
\begin{tabular}{c|cc|cc|cc|cc}
\hline
 &  \multicolumn{2}{c|}{$r=2$} & \multicolumn{2}{c|}{$r=3$}
 & \multicolumn{2}{c|}{$r=4$} & \multicolumn{2}{c}{$r=5$} \\
 $n$ & error & rate & error & rate & error & rate & error & rate \\
\hline
\multicolumn{9}{c}{$\myStrut\Po_{r,r}$ on $\cT_h^2$ meshes}\\
\hline
  \08 & 3.329e-04 & 2.99 & 9.740e-06 & 3.99 & 2.382e-07 & 4.99 & 5.076e-09 & 5.99 \\
   12 & 9.888e-05 & 2.99 & 1.928e-06 & 3.99 & 3.142e-08 & 5.00 & 4.462e-10 & 6.00 \\
   16 & 4.176e-05 & 3.00 & 6.107e-07 & 4.00 & 7.459e-09 & 5.00 & 7.946e-11 & 6.00 \\
   24 & 1.238e-05 & 3.00 & 1.207e-07 & 4.00 & 9.827e-10 & 5.00 & 6.979e-12 & 6.00 \\
\hline
\multicolumn{9}{c}{$\myStrut \cS_{r}$ on $\cT_h^2$ meshes}\\
\hline
  \08 & 5.714e-04 & 2.92 & 4.844e-04 & 2.89 & 2.612e-05 & 3.72 & 2.005e-06 & 4.13 \\
   12 & 1.731e-04 & 2.94 & 1.482e-04 & 2.92 & 6.084e-06 & 3.59 & 3.884e-07 & 4.05 \\
   16 & 7.409e-05 & 2.95 & 6.383e-05 & 2.93 & 2.265e-06 & 3.43 & 1.234e-07 & 3.99 \\
   24 & 2.254e-05 & 2.94 & 1.963e-05 & 2.91 & 5.984e-07 & 3.28 & 2.516e-08 & 3.92 \\
   32 & 9.799e-06 & 2.90 & 8.635e-06 & 2.85 & 2.408e-07 & 3.16 & 8.342e-09 & 3.84 \\
   64 & 1.440e-06 & 2.70 & 1.332e-06 & 2.61 & 2.862e-08 & 3.05 & 6.644e-10 & 3.56 \\
\hline
\multicolumn{9}{c}{$\myStrut \cDS_{r}$ on $\cT_h^2$ meshes}\\
\hline
  \08 & 3.492e-04 & 3.00 & 3.897e-05 & 4.07 & 2.187e-06 & 5.00 & 8.896e-08 & 5.96 \\
   12 & 1.036e-04 & 3.00 & 7.457e-06 & 4.08 & 2.889e-07 & 4.99 & 7.870e-09 & 5.98 \\
   16 & 4.373e-05 & 3.00 & 2.313e-06 & 4.07 & 6.868e-08 & 4.99 & 1.404e-09 & 5.99 \\
   24 & 1.296e-05 & 3.00 & 4.469e-07 & 4.05 & 9.058e-09 & 5.00 & 1.235e-10 & 6.00 \\
\hline
\end{tabular}
\end{center}
\end{table}

\begin{table}[ht]
\caption{$H^1$-seminorm errors and convergence rates for
$\Po_{r,r}$, $\cDS_r$, and $\cS_r$ spaces on trapezoidal meshes.}\label{tab:t2-h1}
\begin{center}
\begin{tabular}{c|cc|cc|cc|cc}
\hline
 &  \multicolumn{2}{c|}{$r=2$} & \multicolumn{2}{c|}{$r=3$}
 & \multicolumn{2}{c|}{$r=4$} & \multicolumn{2}{c}{$r=5$} \\
 $n$ & error & rate & error & rate & error & rate & error & rate \\
\hline
\multicolumn{9}{c}{$\myStrut\Po_{r,r}$ on $\cT_h^2$ meshes}\\
\hline
  \08 & 1.734e-02 & 2.00 & 7.206e-04 & 2.99 & 2.310e-05 & 3.99 & 6.083e-07 & 4.99 \\
   12 & 7.710e-03 & 2.00 & 2.139e-04 & 3.00 & 4.570e-06 & 4.00 & 8.021e-08 & 5.00 \\
   16 & 4.337e-03 & 2.00 & 9.027e-05 & 3.00 & 1.447e-06 & 4.00 & 1.904e-08 & 5.00 \\
   24 & 1.928e-03 & 2.00 & 2.676e-05 & 3.00 & 2.859e-07 & 4.00 & 2.509e-09 & 5.00 \\
\hline
\multicolumn{9}{c}{$\myStrut \cS_{r}$ on $\cT_h^2$ meshes}\\
\hline
  \08 & 2.413e-02 & 1.94 & 1.834e-02 & 1.90 & 1.818e-03 & 2.65 & 1.537e-04 & 3.18 \\
   12 & 1.105e-02 & 1.93 & 8.572e-03 & 1.88 & 6.582e-04 & 2.51 & 4.483e-05 & 3.04 \\
   16 & 6.432e-03 & 1.88 & 5.091e-03 & 1.81 & 3.345e-04 & 2.35 & 1.945e-05 & 2.90 \\
   24 & 3.104e-03 & 1.80 & 2.560e-03 & 1.70 & 1.360e-04 & 2.22 & 6.370e-06 & 2.75 \\
   32 & 1.920e-03 & 1.67 & 1.643e-03 & 1.54 & 7.378e-05 & 2.12 & 3.029e-06 & 2.58 \\
   64 & 7.097e-04 & 1.34 & 6.602e-04 & 1.23 & 1.776e-05 & 2.03 & 5.953e-07 & 2.26 \\
\hline
\multicolumn{9}{c}{$\myStrut \cDS_{r}$ on $\cT_h^2$ meshes}\\
\hline
  \08 & 1.836e-02 & 2.01 & 2.517e-03 & 3.02 & 1.625e-04 & 3.99 & 7.384e-06 & 4.99 \\
   12 & 8.143e-03 & 2.00 & 7.400e-04 & 3.02 & 3.216e-05 & 4.00 & 9.757e-07 & 4.99 \\
   16 & 4.577e-03 & 2.00 & 3.109e-04 & 3.01 & 1.018e-05 & 4.00 & 2.318e-07 & 5.00 \\
   24 & 2.033e-03 & 2.00 & 9.170e-05 & 3.01 & 2.012e-06 & 4.00 & 3.056e-08 & 5.00 \\
\hline
\end{tabular}
\end{center}
\end{table}

We remark that the time cost for the assembly routine can be scaled nearly perfectly in parallel,
since it basically involves only local computations. Therefore, reducing the global number of
degrees of freedom in a serendipity space versus a tensor product space, even perhaps at the expense
of a slightly more expensive assembly, is worthwhile~\cite{Tao_2017_phd,
  Arbogast_Tao_2017_serendipity}.


\subsection{Fully direct mixed finite elements on quadrilaterals}

In this section, we verify the convergence rate for the new fully direct mixed finite elements
derived in Section~\ref{sec:newMixed}. These are implemented without the use of any mapping from the
reference element. We take $\xi_V=\xi_H=\eta_V=\eta_H=1$, although taking the values in
\eqref{def:xi-eta-vh} provides similar results. We apply the hybrid form of the the mixed finite
element method~\cite{Arnold_Brezzi_1985}.  The errors and the orders of convergence for the reduced
and full $H(\Div)$-approximation spaces when $r=1,2$ on mesh $\cT_h^2$ are presented in
Table~\ref{tab:at1-h2}.  Again, results are similar on $\cT_h^3$ meshes. As the theory predicts, the scalar $p$, the vector
$\u$, and the divergence $\div\u$ retain $r$-th, $(r+1)$-st, and $r$-th order approximation,
respectively, for the reduced $H(\Div)$-approximation spaces, and all three quantities show $r$-th order
approximation for the full $H(\Div)$-approximation spaces.

\begin{table}[ht]
\caption{Errors and convergence rates for
fully direct mixed spaces on trapezoidal meshes $\cT_h^2$.}\label{tab:at1-h2}
\begin{center}
\begin{tabular}{c|cc|cc|cc}
\hline
 &  \multicolumn{2}{c|}{$||p-p_h||$} & \multicolumn{2}{c|}{$\|\u-\u_h\|$}
 & \multicolumn{2}{c}{$\|\div(\u-\u_h)\|$}\\
 $n$ & error & rate & error & rate & error & rate\\
\hline
\multicolumn{7}{c}{$r=1$, reduced $H(\Div)$-approximation}\\
\hline
  \04 & 1.670e-01 &    --- & 2.609e-01 &    --- & 3.163e-00 &    --- \\ 
  \08 & 8.271e-02 & 1.01 & 6.803e-02 & 1.96 & 1.612e-00 & 0.98 \\ 
   16 & 4.117e-02 & 1.00 & 1.719e-02 & 1.99 & 8.099e-01 & 1.00 \\ 
   32 & 2.056e-02 & 1.00 & 4.309e-03 & 2.00 & 4.054e-01 & 1.00 \\ 
\hline
\multicolumn{7}{c}{$r=2$, reduced $H(\Div)$-approximation}\\
\hline
  \04 & 3.079e-02 &    --- & 2.319e-02 &    --- & 6.067e-01 &    --- \\ 
  \08 & 7.847e-03 & 1.98 & 2.906e-03 & 3.00 & 1.549e-01 & 1.98 \\ 
   16 & 1.972e-03 & 2.00 & 3.633e-04 & 3.00 & 3.892e-02 & 2.00 \\ 
   32 & 4.936e-04 & 2.00 & 4.543e-05 & 3.00 & 9.742e-03 & 2.00 \\ 
\hline
\multicolumn{7}{c}{$r=1$, full $H(\Div)$-approximation}\\
\hline
  \04 & 3.079e-02 &    - & 5.562e-02 &    - & 6.067e-01 &    - \\ 
  \08 & 7.847e-03 & 1.98 & 1.350e-02 & 2.02 & 1.549e-01 & 1.98 \\ 
   16 & 1.972e-03 & 2.00 & 3.355e-03 & 2.01 & 3.892e-02 & 2.00 \\ 
   32 & 4.936e-04 & 2.00 & 8.378e-04 & 2.00 & 9.742e-03 & 2.00 \\ 
\hline
\multicolumn{7}{c}{$r=2$, full $H(\Div)$-approximation}\\
\hline
  \04 & 4.081e-03 &    - & 7.198e-03 &    - & 8.050e-02 &    - \\ 
  \08 & 5.201e-04 & 2.98 & 9.105e-04 & 2.99 & 1.026e-02 & 2.98 \\ 
   16 & 6.533e-05 & 3.00 & 1.141e-04 & 3.00 & 1.289e-03 & 3.00 \\ 
   32 & 8.176e-06 & 3.00 & 1.428e-05 & 3.00 & 1.614e-04 & 3.00 \\ 
\hline
\end{tabular}
\end{center}
\end{table}

Our numerical test agrees with that taken in~\cite{Arbogast_Correa_2016}, where results for the
full and reduced AC spaces and the mapped BDM spaces appear.  Results for our fully direct mixed
spaces agree very closely with the results for the AC spaces, and these far exceed the
performance of the mapped BDM spaces.


\subsection{Serendipity space based on mapped supplements}

In this section, we present the errors and convergence rates for the serendipity spaces
$\cDS_r^{\text{map}}$ using elements defined in \eqref{eq:DSmap}, which has supplements mapped from
the reference element. The results for $r=2,3,4$ on mesh $\cT_h^2$ are shown in
Table~\ref{tab:ds-mapped}.  As predicted by the theory, this new family of spaces shows an
$(r+1)$-st order convergence in the $L^2$-norm and an $r$-th order convergence in the
$H^1$-seminorm.  The results compare favorably with those for the fully direct spaces in
Tables~\ref{tab:t2-l2}--\ref{tab:t2-h1}, although the latter are perhaps slightly better.

\begin{table}[ht]
\caption{Errors and convergence rates for
$\cDS_r^{\textrm{map}}$ spaces on trapezoidal meshes $\cT_h^2$.}\label{tab:ds-mapped}
\begin{center}
\begin{tabular}{c|cc|cc|cc|cc}
\hline
 &  \multicolumn{2}{c|}{$r=2$} & \multicolumn{2}{c|}{$r=3$}
 & \multicolumn{2}{c|}{$r=4$} & \multicolumn{2}{c}{$r=5$}\\
 $n$ & error & rate & error & rate & error & rate & error & rate \\
\hline
\multicolumn{9}{c}{$\myStrut L^2$-errors and convergence rates}\\
\hline
  \08 & 5.737e-04 & 2.92 & 4.128e-05 & 4.09 & 2.344e-06 & 5.04 & 9.134e-08 & 6.00 \\
   12 & 1.727e-04 & 2.96 & 7.968e-06 & 4.06 & 3.048e-07 & 5.03 & 8.023e-09 & 6.00 \\
   16 & 7.329e-05 & 2.98 & 2.493e-06 & 4.04 & 7.182e-08 & 5.03 & 1.428e-09 & 6.00 \\
   24 & 2.180e-05 & 2.99 & 4.869e-07 & 4.03 & 9.380e-09 & 5.02 & 1.252e-10 & 6.00 \\
\hline
\multicolumn{9}{c}{$\myStrut H^1$-seminorm errors and convergence rates}\\
\hline
  \08 & 2.410e-02 & 1.99 & 2.851e-03 & 3.05 & 1.730e-04 & 4.03 & 7.609e-06 & 5.01 \\
   12 & 1.074e-02 & 1.99 & 8.333e-04 & 3.03 & 3.385e-05 & 4.02 & 9.979e-07 & 5.01 \\
   16 & 6.047e-03 & 2.00 & 3.491e-04 & 3.02 & 1.065e-05 & 4.02 & 2.362e-07 & 5.01 \\
   24 & 2.690e-03 & 2.00 & 1.027e-04 & 3.02 & 2.091e-06 & 4.02 & 3.102e-08 & 5.01 \\
\hline
\end{tabular}
\end{center}
\end{table}


\section{Summary and Conclusions}\label{sec:conc}

It is possible to define a wide variety of direct serendipity elements on a nondegenerate, convex
quadrilateral $E$.  Most or perhaps all of these elements appear to be new, and they have the form
\begin{equation}\label{eq:serendipityFormConc}
\cDS_r(E) = \Po_r(E)\oplus\Supp_r^\cDS(E),\quad r\geq2.
\end{equation}
The supplemental space $\Supp_r^\cDS(E)$ can be defined by four functions. Referring to
Figure~\ref{fig:numbering}, the linear functions $\lambda_H$ and $\lambda_V$ are arbitrary except
that the zero line of $\lambda_H$ must intersect the lines containing $e_1$ and $e_2$ above or below
the intersection point $\x_{12}$, if it exists, and $\lambda_V$ must intersect the lines containing
$e_3$ and $e_4$ to the left or right of the intersection point $\x_{34}$, if it exists.  The
bounded, (most likely) nonlinear functions $R_V$ and $R_H$ can be chosen arbitrarily as long as they
are negative constants on $e_1$ and $e_3$, respectively, and positive constants on $e_2$ and $e_4$,
respectively.  For example, one can take the simple choices
\begin{equation}
\lambda_H = \lambda_3 - \lambda_4,\quad R_V = \frac{\lambda_1 - \lambda_2}{\lambda_1 + \lambda_2},\quad
\lambda_V = \lambda_1 - \lambda_2,\quad R_H = \frac{\lambda_3 - \lambda_4}{\lambda_3 + \lambda_4},
\end{equation}
or the choices given in Lemma~\ref{lem:equ-space}, in which case the explicit basis
\eqref{eq:ds-baseE1}--\eqref{eq:ds-baseV13} can be constructed easily. The fully direct supplemental
space is
\begin{equation}
\label{eq:supplementSpaceConc}
\Supp_r^{\cDS}(E)
= \spn\{\lambda_3\lambda_4\lambda_H^{r-2}R_V,\lambda_1\lambda_2\lambda_V^{r-2}R_H\},
\end{equation}
but a supplemental space can also be defined using the bilinear map between $\hat E=[-1,1]^2$ and $E$ as
\begin{equation}
\label{eq:supplementSpaceConcMapped}
\Supp_{r}^{\cDS,\text{map}}(E) = \spn\big\{F_{\!E}^0\big((1-\hat x_2^2)\hat x_1\hat x_2^{r-2}\big),
                                                F_{\!E}^0\big((1-\hat x_1^2)\hat x_2\hat x_1^{r-2}\big)\big\}.
\end{equation}

It is possible to define a wide variety of direct mixed elements on $E$. The de Rham theory is
useful in this regard, and the elements take the form
\begin{align}
\V_r^{\textrm{red}}(E) &= \Curl\,\cDS_{r+1}(E)\oplus\x\Po_{r-1}(E) = \Po_r^2(E)\oplus\Supp_r^\V(E)\quad r\geq1,\\
\V_r^{\textrm{full}}(E) &= \V_r^{\textrm{red}}(E)\oplus\x\tilde\Po_r(E),\quad r\geq1,
\end{align}
for reduced and full $H(\Div)$-approximation spaces, where
\begin{equation}
\Supp_r^\V(E) = \Curl\,\Supp_r^\cDS(E).
\end{equation}
If \eqref{eq:supplementSpaceConcMapped} is used, the AC spaces \cite{Arbogast_Correa_2016} result.
Otherwise, the elements appear to be new, and they are the first families of fully direct mixed
spaces defined on quadrilaterals.

The direct serendipity and mixed elements can be merged to create $H^1(\Omega)$ and $H(\Div;\Omega)$
conforming spaces, respectively, on a mesh $\cT_h$ of nondegenerate, convex quadrilaterals.  If the
meshes are shape regular as $h\to0$, the spaces have both optimal approximation properties and
minimal local dimension.  Numerical results were presented to illustrate their performance.

We close with a simple observation. Another well-known de Rham complex is
\begin{equation}\label{eq:deRhamCurl}
\Re \hooklongrightarrow H^1 \overset{\Grad\,}{\longlongrightarrow}
H(\Curl) \overset{\Curl\,}{\longlongrightarrow} L^2 \longrightarrow 0,
\end{equation}
where the curl of a vector function $\bfpsi=\big(\psi_1,\psi_2\big)$ is the scalar
$\Curl\,\bfpsi(\x) = \dfrac{\d\psi_1}{\d x_2} - \dfrac{\d\psi_2}{\d x_1}$.  This complex is
essentially just a rotation of \eqref{eq:deRham}.  It gives us full ($s=r$) and reduced ($s=r-1$)
direct $H(\Curl)$-approximating elements
\begin{equation}
\label{eq:generalHCurl}
\V_{r,\Curl}(E) = \grad\,\cDS_{r+1}(E)\oplus(x_2,-x_1)\Po_{s}.
\end{equation}
These then satisfy the de Rham complexes
\begin{align}
\label{eq:deRhamVCurl}
\Re \hooklongrightarrow \cDS_{r+1}(E) \overset{\Grad\,}{\longlongrightarrow}
 &\V_{r,\Curl}^{\textrm{full}}(E) \overset{\Curl\,}{\longlongrightarrow} \Po_{r}(E) \longrightarrow 0,\\
\label{eq:deRhamVredCurl}
\Re \hooklongrightarrow \cDS_{r+1}(E) \overset{\Grad\,}{\longlongrightarrow}
 &\V_{r,\Curl}^{\textrm{red}}(E) \overset{\Curl\,}{\longlongrightarrow} \Po_{r-1}(E) \longrightarrow 0,
\end{align}
for any variant of our new direct serendipity spaces.  We can merge these elements globally, since
tangential derivatives of $\phi$, say, map to tangential components of $\grad\phi$, which are both
$\grad\phi\cdot\tau$. So we have full and reduced direct $H(\Curl)$-approximating spaces over meshes
of quadrilaterals.



\end{document}